\documentclass{article}
\usepackage{amsmath,amsthm,amsfonts}
\usepackage{graphicx}
\usepackage{epstopdf}
\usepackage{multirow}
\def\smallddots{\mathinner{\raise7pt\hbox{.}\raise4pt\hbox{.}\raise1pt\hbox{.}}}
\def\smallsdots{\mathinner{\raise1pt\hbox{.}\raise4pt\hbox{.}\raise7pt\hbox{.}}}

\DeclareMathOperator{\diag}{diag}

\DeclareMathOperator{\rank}{rank}
\DeclareMathOperator{\nrank}{nrank}

\numberwithin{equation}{section}
\numberwithin{table}{section}
\newtheorem{theorem}{Theorem}[section]
\newtheorem{lemma}{Lemma}[section]

\newtheorem{corollary}{Corollary}[section]
\newtheorem{fact}{Fact}[section]

\newtheorem{algorithm}{Algorithm}[section]
\newtheorem{example}{Example}[section]
\newtheorem{definition}{Definition}[section]

\newtheorem{remark}{Remark}[section]

\begin{document}

\title{\bf  Random Multipliers Numerically Stabilize Gaussian and Block Gaussian Elimination: 
Proofs and an Extension to Low-rank Approximation
\thanks {Some results of this paper have been presented at the 
 ACM-SIGSAM International 
Symposium on Symbolic and Algebraic Computation (ISSAC '2011), in San Jose, CA, 
June 8--11, 2011,
the
3nd International Conference on Matrix Methods in Mathematics and 
Applications (MMMA 2011), in
Moscow, Russia, June 22--25, 2011, 
the 7th International Congress on Industrial and Applied Mathematics 
(ICIAM 2011), in Vancouver, British Columbia, Canada, July 18--22, 
2011,
the SIAM International Conference on Linear Algebra,
in Valencia, Spain, June 18-22, 2012,  
the Conference on Structured Linear and Multilinear Algebra Problems (SLA2012),
in  Leuven, Belgium, September 10-14, 2012, and the 18th Conference of the 
International Linear Algebra Society (ILAS 2013), in Providence, RI, USA, June 3--7, 2013}}

\author{Victor Y. Pan$^{[1, 2],[a]}$, Guoliang Qian$^{[2],[b]}$ and Xiaodong Yan $^{[2],[c]}$
%, and Ai-Long Zheng$^{[2],[c]}$
%, and Ai-Long Zheng$^{[2],[c]}$ 
\and\\
$^{[1]}$ Department of Mathematics and Computer Science \\
Lehman College of the City University of New York \\
Bronx, NY 10468 USA \\
$^{[2]}$ Ph.D. Programs in Mathematics  and Computer Science \\
The Graduate Center of the City University of New York \\
New York, NY 10036 USA \\
$^{[a]}$ victor.pan@lehman.cuny.edu \\
http://comet.lehman.cuny.edu/vpan/  \\
$^{[b]}$ gqian@gc.cuny.edu \\
$^{[c]}$ xyannyc@yahoo.com \\
} 

%------------------------------------------------------------------------------

\maketitle

%------------------------------------------------------------------------------

\begin{abstract}
%\noindent 
We study two applications of standard Gaussian random multipliers. At 
first we prove that with a probability close to 1 such a multiplier is 
expected to numerically stabilize Gaussian elimination with no pivoting
as well as block Gaussian elimination. Then, by extending our analysis, 
we prove that such a multiplier is also expected to support low-rank 
approximation of a matrix without customary oversampling. Our test 
results are in good accordance with this formal study. The results 
remain similar when we replace Gaussian multipliers with random circulant 
or Toeplitz multipliers, which involve fewer random parameters and enable 
faster multiplication. We formally support the observed efficiency of random 
structured multipliers  applied to approximation, but not to elimination. 
Moreover, we prove that with a probability close to 1 Gaussian random 
circulant multipliers do not fix numerical instability of the elimination 
algorithms for a specific narrow class of well-conditioned inputs. We know 
of no such hard input classes for various alternative choices of random 
structured multipliers, but for none of such multipliers
we have a formal proof of its efficiency  
 for numerical Gaussian elimination.
\end{abstract}


 
%Linear systems of equations

% - - - - - - - - - - - - - - - - - - - - - - - - - - - - - - - - - - - - -

\section{Introduction}\label{sintro}

% - - - - - - - - - - - - - - - - - - - - - - - - - - - - - - - - - - - - -
% - - - - - - - - - - - - - - - - - - - - - - - - - - - - - - - - - - - - -

\subsection{Overview}\label{sover}

% - - - - - - - - - - - - - - - - - - - - - - - - - - - - - - - - - - - - -

We call a standard Gaussian random matrix
 just {\em Gaussian}
and  
apply Gaussian multipliers to

\begin{itemize}

\item %1
 numerically stabilize
 Gaussian 
and block Gaussian elimination
 using no pivoting, orthogonalization
or symmetrization
\item %2 
approximate the leading singular spaces
of an ill-con\-di\-tioned matrix, 
 associated with 
its largest singular values 
 and 
\item %3
approximate this matrix by a low-rank matrix.
\end{itemize}

%It is well known that a Gaussian matrix is 
%expected to be well-con\-di\-tioned
% \cite{D88},  \cite{E88},  \cite{ES05},
%\cite{CD05}, \cite{SST06}, but does this imply the power
%of Gaussian multipliers in these applications?
Ample empirical evidence shows efficiency of all these applications
 \cite{HMT11}, \cite{M11}, \cite{PQZ13}, 
but formal proofs are only known
in the case of approximation  
under the assumption of small oversampling,
that is, a minor increase of the multiplier size  \cite{HMT11}.
We provide 
formal support for randomized stabilization of Gaussian and block Gaussian 
elimination  and for
randomized approximation without oversampling. 

Our test results are in good accordance with our analysis
and also show similar power of  random circulant
multipliers or their leading (that is, northeastern) Toeplitz blocks 
for the   
elimination and approximation, respectively.  
Simple extension of the known 
estimates enables formal support for this empirical observation
in the case of 
low-rank approximation, but not
in the case of elimination.
Moreover, we prove that with a high probability 
Gaussian circulant multipliers cannot fix numerical instability
 of the elimination algorithms 
for a specific narrow class of inputs.
The issue remains open for the elimination preprocessed 
with other random structured multipliers.

In the next subsections we specify further the computational tasks and our results.

% - - - - - - - - - - - - - - - - - - - - - - - - - - - - - - - - - - - - -

\subsection{Numerical Gaussian elimination with no pivoting and 
block Gaussian elimination}\label{sgenp}

% - - - - - - - - - - - - - - - - - - - - - - - - - - - - - - - - - - - - -
 
Gaussian elimination, applied
numerically,
with rounding errors, can 
  fail even in the case of 
a nonsingular well-con\-di\-tioned input matrix
unless
this matrix is also
positive definite,
 diagonally 
dominant,  or 
totally positive.
For example (see part (i) of Theorem \ref{thcgenp}), 
numerical Gaussian elimination fails when it is 
applied to
the  matrices of discrete Fourier transform of large sizes,
even though they  are unitary up to scaling.
In practice the
user avoids the problems by applying 
 Gaussian elimination 
with partial pivoting,
that is, appropriate row interchange,
which has some limited 
formal and ample empirical support. 
Our alternative is  
Gaussian elimination with no pivoting.
Hereafter we use the acronyms  {\em GEPP} and 
{\em  GENP}.

In Section \ref{ssgnp} we prove that 
GENP is safe numerically
with a probability close to 1
 when the input matrix
 is  
nonsingular, well-conditioned and preprocessed
with a Gaussian 
%and  random circulant 
 multiplier. 

At the first glance our preprocessing may seem to be more costly than
partial pivoting. The latter only involves order of $n^2$
comparisons for an $n\times n$ input matrix $A$, versus 
$2n^3-n^2$ arithmetic operations for its multiplication 
by a Gaussian multiplier.
Generation of $n^2$ independent Gaussian 
entries of the  $n\times n$ multiplier
is  an additional charge, even though 
 one can discount partly it because this  
preprocessing stage is independent of the input matrix.
More careful comparison, however, shows that
partial pivoting takes quite a heavy toll.
%"pivoting usually degrades the performance"
%\cite[page 119]{GL96}.
It 
interrupts the stream of arithmetic operations 
with foreign operations of comparison,
involves book-keeping, compromises data locality, and
increases communication overhead and data dependence. 
%tends to destroy matrix structure (see Section \ref{srsm}).
 
Choosing 
between GEPP and GENP with randomized preprocessing 
%requires further study beyond the scope of this paper and must be 
%left to 
the user may also consider various other factors,  
 some of them dynamic in time.
For example, here is a relevant citation from  \cite{BCD14}:
``The traditional metric for the efficiency of a numerical algorithm has been the number of arithmetic operations it performs. Technological trends have long been reducing the time to perform an arithmetic operation, so it is no longer the bottleneck in many algorithms; rather, communication, or moving data, is the bottleneck".

Our  test results with Gaussian input  matrices
are in good accordance with our formal estimates
 (see Figures \ref{fig:fig1}--\ref{fig:fig5} 
 and Tables \ref{tab61}--\ref{tab65a}). 
%even where we {\em filled them with
%the integers} $\pm 1$, thus
%limiting randomization to the choice of 
%the $n$ signs $+$ or $-$
%in the case of an $n\times n$ input.
In our tests  
the output accuracy of GENP with preprocessing
was a little lower than in the case of the customary GEPP,
as can be expected,
but 
 a single step of 
iterative refinement,
performed at a dominated 
computational cost, has always fixed this discrepancy
(see Figures \ref{fig:fig2} and \ref{fig:fig3} and  Table \ref{tab63}). 
The tests show similar results when we applied Gaussian circulant
rather than general Gaussian
multipliers, and in this case the cost of our preprocessing 
decreases dramatically
(see more on that in Section \ref{srsm}).

Finally, 
all our study of GENP,  including formal 
probabilistic support of its numerical performance,
is immediately extended to {\em block Gaussian elimination}
(see Section \ref{sbgegenpn}),
whereas  pivoting cannot amend this valuable algorithm,
and so without our preprocessing, it is numerically unsafe to use it,
unless an  input matrix is nonsingular, well-con\-di\-tioned, and
positive definite,
 diagonally 
dominant,  or 
totally positive.
 
% - - - - - - - - - - - - - - - - - - - - - - - - - - - - - - - - - - - - -

\subsection{Low-rank approximation of a matrix}\label{srdpr}

% - - - - - - - - - - - - - - - - - - - - - - - - - - - - - - - - - - - - -

Random multipliers are known to be highly
efficient  
 for low-rank approximations
of an $m\times n$ matrix $A$ having a small numerical rank $r$.
As the basic step, one computes the product $AH$ where $H$ is 
a random $n\times l$ multiplier, for  $l=r+p$ and 
a positive {\em oversampling integer} $p$.
The resulting randomized algorithms  have been studied
extensively, both formally and experimentally
 \cite{HMT11},  \cite{M11}.
They
are
numerically stable, run at a low 
computational cost, 
allow low-cost improvement of 
the output accuracy by means of the Power Method,
and have 
 important applications to
matrix computations, data mining,
statistics, PDEs and integral equations.

By extending our analysis  of preprocessed GENP, 
we prove that even for an $n\times r$ Gaussian
multiplier, that is, {\em without oversampling},
the algorithms output 
rank-$r$  approximations of the input matrix
with a probability close to 1.
Then again our test results 
are in good accordance with our formal estimates
 (see Figures \ref{fig:fig6} 
and \ref{fig:fig7} and Tables \ref{tab67}--\ref{tab614}).

The decrease of $p$ to 0 should be theoretically interesting,
although it has only  minor practical promise, 
limited to the cases where   
the numerical rank $r$ is small and  
the user knows it. Indeed, according to \cite[Section 4.2]{HMT11},
``it is adequate to choose  $\dots~~p=5$ or $p=10$".

\subsection{Computations with random structured multipliers}\label{srsm}

The SRFT $n\times l$  multipliers $H$ 
(SRFT is the acronym for Subsample Random Fourier Transform)
involve only $n$ random parameters
versus $nl$ parameters of Gaussian multipliers and accelerate 
the computation of the product $AH$ by a factor of $l/\log (l)$
 versus $n\times l$ Gaussian multipliers $H$.
It has been proved that they are expected to support rank-$r$ approximation
 assuming
oversampling integers $p=l-r$ of order $r\log (r)$,
and empirically this has been observed for reasonably small constants $p$, usually 
being not more than 20
 \cite[Section 11]{HMT11}, \cite{M11}.
We readily extend the latter results
 to the case where $n\times l$ products of random  $n\times n$
circulant and random $n\times l$ permutation matrices
are used as multipliers instead of SRFT matrices (see Remark \ref{remsrtft}),
and we observe such empirical behavior also in the cases when we 
preprocess GENP by applying random circulant  multipliers
(see  Figures  \ref{fig:fig2},  \ref{fig:fig3}, \ref{fig:fig6} 
and \ref{fig:fig7} and Tables \ref{tab64}, \ref{tab65}, \ref{tab611}, and \ref{tab614}).
In the latter case we only need $n$
random parameters versus $n^2$ for a Gaussian multiplier
and accelerate Gaussian preprocessing  for GENP by a factor of $n/\log (n)$.
  
This acceleration factor grows to $n^2/\log (n)$ 
if an input matrix is Toeplitz or Toeplitz-like
and if we can apply the
 MBA  celebrated 
 algorithm, which is just recursive block
 Gaussian elimination adjusted to a Toeplitz-like input \cite[Chapter 5]{p01}
and which is superfast, that is, runs in nearly linear arithmetic time.
Numerical stability problems, however, are well known for this algorithm
\cite{B85}, and fixing them was the central subject of the highly recognized 
papers \cite{GKO95} and \cite{G98}. Pivoting could not be applied here
because it destroys Toeplitz structure, thus increasing the solution cost to cubic.
So the authors first reduced the task to the  case of Cauchy-like inputs by
specializing the techniques of the transformation of matrix structures from \cite{P90} 
(also see \cite{P15})
and then applied a fast Cauchy-like variant of GEPP
using quadratic arithmetic time. 
 
Stabilization of the MBA superfast algorithm by means  of random
circulant multipliers would mean randomized acceleration 
by order of magnitude versus   \cite{GKO95} and \cite{G98}
because circulant multipliers preserve Toeplitz-like structure
of an input matrix and thus keep the MBA algorithm superfast.
Empirically, Gaussian circulant multipliers do stabilize
GENP numerically,
but no formal support is known for  this observed behavior.
Moreover we
even prove that 
the application  of GENP 
 fails
 for a specific narrow class of  matrices, 
and with probability near 1
 their preprocessing with Gaussian circulant multipliers 
does not help 
 (see Remark \ref{remsrtft}). 
We know of  no such hard inputs
 for some variations 
of using Gaussian circulant multipliers for GENP and MBA, such as
using products of random circulant and skew-circulant multipliers
and using random circulant pre-multipliers and post-multipliers simultaneously.
Both variations preserve Toeplitz-like input structure 
and allow superfast performance of the MBA algorithm,
but we have no formal support for the efficiency of such multipliers.

\subsection{Related works}\label{srel}

Preconditioning of linear systems of equations  
is a classical subject \cite{A94}, \cite{B02}, \cite{G97}.
For early work on randomized multiplicative preprocessing
as a means of countering degeneracy of matrices
see 
Section 2.13 ``Regularization of a Matrix via 
Preconditioning with Randomization" in \cite{BP94}
 and the bibliography therein. On the most recent advance in this direction see \cite{PZa}.
On the early specialization of such techniques to Gaussian elimination
see \cite{PP95}.
Randomized multiplicative preconditioning
for numerical stabilization of GENP 
was proposed in \cite[Section 12.2]{PGMQ} 
and \cite{PQZ13}, although only
weaker theorems on the
 formal support of this approach
were stated and their proofs were omitted.
The paper \cite{BBD12}
and the bibliography therein cover the
heuristic application of PRBMs
(that is, Partial Random Butterfly Multipliers), providing 
some empirical support 
 for GENP with  preprocessing.
On low-rank approximation
we refer the reader to the surveys \cite{HMT11} and \cite{M11},
which were the springboard for our study in Section \ref{sapsr}.
We cite these and other related works throughout 
the paper and refer the reader to \cite[Section 11]{PQZa}
for further bibliography. 
The estimates of our Corollary \ref{cor1} 
are close to the ones of  \cite[Theorem 3.8]{PQ10},
which
were 
the basis for our algorithms in \cite{PQ10}, 
\cite{PQ12}, and \cite{PQZC}.  
Unlike the latter papers, however, we 
state these basic estimates in a simpler
form, refine them by following 
 \cite{CD05} rather than \cite{SST06},
and include their detailed proofs.
On the related subject of estimating the norms and condition
numbers of Gaussian matrices and random structured matrices 
see   \cite{D88},  \cite{E88},  \cite{ES05},
\cite{CD05}, \cite{SST06}, \cite{HMT11}, \cite{T11}, and \cite{PSZa}. 
For a natural extension
of our present work,
 one can combine
randomized matrix multiplication with
randomized augmentation and 
 additive preprocessing
of   \cite{PQ10}, \cite{PQ12}, and
 \cite{PQZC}. 

% - - - - - - - - - - - - - - - - - - - - - - - - - - - - - - - - - - - - -

\subsection{Organization of the paper}\label{sorg}

% - - - - - - - - - - - - - - - - - - - - - - - - - - - - - - - - - - - - -

%We organize the paper as follows.
%on matrix computations
In the next 
section
we recall some definitions and basic results.
In Section \ref{sbgegenpn} we show that 
GENP and block Gaussian elimination
are numerically safe for a matrix whose 
all leading blocks are nonsingular 
and well conditioned. 
In Section \ref{smrc} we  estimate  
the impact of  
 preprocessing
with general nonrandom multipliers on 
these properties of the leading blocks.
In Sections \ref{sapsr1}
and
\ref{srnd} we extend our analysis 
from  Section \ref{smrc} in order to cover the impact of 
Gaussian and random structured multipliers,
respectively.
%and for a number of other fundamental 
%matrix computations. We cover some of them in 
In Section \ref{sapsr}
 we recall  an algorithm from  \cite{HMT11} for
low-rank approximation and prove that this 
randomized algorithm is expected to
work even with no oversampling.
 In Section \ref{sexp} we cover numerical tests
%which constitute 
(the contribution of the last two authors).
Section \ref{sconcl} contains a brief summary.
%In Section \ref{snewt} we briefly recall Newton's iteration 
%for the inversion of structured matrices,
% refine it with our preprocessing 
%and discuss its heuristic acceleration.
%In  Appendix \ref{savrnk} we estimate the dimension of the variety 
%of the matrices of smaller ranks. 
In  Appendix \ref{sssm} we recall the known
probabilistic estimates for the error norms of 
randomized low-rank approximations.
%for the sake of completeness of our exposition.
In  Appendix \ref{srsnrm}
we estimate the probability that 
 a 
 random matrix 
has full rank
under the uniform probability distribution.
In Appendix \ref{sosvdi0}
we  estimate the perturbation errors of matrix inversion.
In Appendix \ref{stab} we display tables with our test results,
which are more detailed than the data given by the plots in Section  \ref{sexp}.  
Some readers may be only interested in 
the part of our paper on GENP.
They can skip Sections 
 \ref{sapsr}
 and \ref{stails}.

%------------------------------------------------------------------------------
%------------------------------------------------------------------------------

\section{Some definitions}\label{sdef}

%------------------------------------------------------------------------------

Except for using unitary circulant matrices in Sections  \ref{srnd} and
\ref{stails}, we
assume computations in the field $\mathbb R$ of real numbers,
but the extension to the case of the complex field $\mathbb C$
is quite straightforward.
%(cf. \cite{CD05}).
Hereafter ``flop" stands for ``arithmetic operation",
`` i.i.d." stands for ``independent identically distributed",
and ``Gaussian matrix" stands for
``standard Gaussian random matrix" (cf. Definition \ref{defrndm}). 
 The concepts ``large", ``small", ``near", ``closely approximate", 
``ill-con\-di\-tioned" and ``well-con\-di\-tioned" are 
quantified in the context. By saying ``expect" and ``likely" we
mean ``with probability $1$ or close to $1$".
(We only use the concept of the expected value in Theorem \ref{thszcd}, 
Corollary \ref{cocd}, and Appendix \ref{sssm}.) 
%reproduced from \cite{S91},  \cite[Theorem 6.1]{CD05}, and \cite[Theorems 10.5 and 10.6]{HMT11}.)
%------------------------------------------------------------------------------

Next we recall and extend some customary definitions of matrix computations
\cite{GL13}, \cite{S98}.

%------------------------------------------------------------------------------

%\section{Some definitions and basic  results of matrix computations}\label{smat}

%------------------------------------------------------------------------------
%------------------------------------------------------------------------------

%\subsection{SVD and pseudo inverse}\label{ssvdpsd}
%------------------------------------------------------------------------------

$\mathbb  R^{m\times n}$ is the class of real $m\times n$ matrices $A=(a_{i,j})_{i,j}^{m,n}$.

%------------------------------------------------------------------------------ 

$\diag (B_1,\dots,B_k)=\diag(B_j)_{j=1}^k$ is a $k\times k$ block diagonal matrix 
with the diagonal blocks $B_1,\dots,B_k$. $(B_1)_{1}^k$, 
$(B_1~|~\dots~|~B_k)$, and $(B_1,\dots,B_k)$ 
denote a $1\times k$ block matrix with the blocks $B_1,\dots,B_k$.
In both cases the blocks $B_j$ can be rectangular.
  
$I_n$
is the $n\times n$ identity
matrix. 
$O_{k,l}$ is the $k\times l$ matrix filled with zeros. 
We write $I$ and $O$ if the  matrix size  
 is defined by context. 
%------------------------------------------------------------------------------
%$\mathcal N(A)$ denotes its null space $\{{\bf v}:~A{\bf v}={\bf 0}\}$,
$A^T$ is the transpose  of a 
matrix $A$.

 $A_{k,l}$  denotes its leading, 
that is, northwestern  $k\times l$ block 
submatrix, and 
%in Section \ref{sbgegenpn}
we also write  $A^{(k)}=A_{k,k}$.

%------------------------------------------------------------------------------

$||A||=||A||_2$
is the spectral norm of a matrix $A$.
$||A||_F$
is its Frobenius norm.

A real matrix $Q$ is  
{\em orthogonal} if $Q^TQ=I$ 
 or $QQ^T=I$. 
$(Q,R)=(Q(A),R(A))$ for an $m\times n$ matrix $A$ of rank $n$
denotes a unique pair of orthogonal $m\times n$ and
 upper triangular $n\times n$ matrices such that $A=QR$
and all diagonal entries of the matrix $R$
are positive  \cite[Theorem 5.2.3]{GL13}. 

 $A^+$ denotes the Moore--Penrose 
pseudo-inverse of an $m\times n$ matrix $A$, and
\begin{equation}\label{eqsvd}
A=S_A\Sigma_AT_A^T
\end{equation}
denotes its SVD where $S_A^TS_A=S_AS_A^T=I_m$,
$T_A^TT_A=T_AT_A^T=I_n$,
$\Sigma_A=\diag(\sigma_j(A))_{j}$,  and
%\begin{equation}\label{eqsingv}
$\sigma_j=\sigma_j(A)$
%\end{equation}
is the $j$th largest singular value of $A$.
If a matrix $A$ has full column rank $\rho$, then
% $A^+=(A^TA)^{-1}A^T$ and
%$A^+$ is a left  inverse of a matrix $A$ of full rank for $m\ge n$ and its right inverse for $m\le n$. 
\begin{equation}\label{eqnrm+}
||A^+||=1/\sigma_{\rho}(A).
\end{equation}

  $A^{+T}$ stands for $(A^+)^T=(A^T)^+$, $A_s^T$  for $(A_s)^{T}$,
and $A_s^+$ for $(A_s)^{+}$ where $s$ can denote a scalar, a matrix, or
a pair of such objects, e.g., $A_{k,l}^T$ stands for $(A_{k,l})^T$. 

%We reuse the definitions in the Introduction.

%------------------------------------------------------------------------------

%\subsection{Condition number, numerical rank and  strongly well-conditioned matrices
%and  perturbation norm bounds
%}\label{scnpn}

%------------------------------------------------------------------------------

$\kappa (A)=\frac{\sigma_1(A)}{\sigma_{\rho}(A)}=||A||~||A^+||$ is the condition 
number of an $m\times n$ matrix $A$ of a rank $\rho$. Such matrix is 
{\em ill-con\-di\-tioned}
if the ratio $\sigma_1(A)/\sigma_{\rho}(A)$
is large. If the ratio is reasonably bounded,
then the matrix is {\em well-con\-di\-tioned}.
An $m\times n$ matrix $A$ has a {\em numerical rank} 
$r=\nrank(A)\le \rho=\rank (A)$ 
if the ratios $\sigma_{j}(A)/||A||$
are small for $j>r$ but not for $j\le r$. 

%------------------------------------------------------------------------------
%------------------------------------------------------------------------------

The following concepts 
cover all rectangular matrices, but we need 
them just in the case of square matrices, whose sets of leading blocks 
include the   matrices themselves.
A matrix is {\em strongly nonsingular} if all its leading blocks
 are nonsingular. 
Such a matrix is 
 {\em strongly well-conditioned} if all its leading blocks
 are well-conditioned. 

We recall further relevant definitions and basic results  of
 matrix computations
in the beginning of Section \ref{sapsr} and in the Appendix.

%------------------------------------------------------------------------------

\section{Block Gaussian elimination and GENP
}\label{sbgegenpn}

%------------------------------------------------------------------------------

For a 
nonsingular
$2\times 2$ block matrix $A=\begin{pmatrix}
B  &  C  \\
D  &  E
\end{pmatrix}$ of size $n\times n$
with nonsingular $k\times k$ {\em pivot block} $B=A^{(k)}$, 
 define 
$S=S(A^{(k)},A)=E-DB^{-1}C$,
%=S(A^{(k)},A)=A_{11}-A_{01}A_{00}^{-1}A_{01}$, 
the {\em Schur complement} of $A^{(k)}$ in $A$,
and
the block factorizations,

\begin{equation}\label{eqgenp}
\begin{aligned}
A=\begin{pmatrix}
I_k  &  O_{k,r}  \\
DB^{-1}  & I_r
\end{pmatrix}
\begin{pmatrix}
B  &  O_{k,r} \\
O_{r,k}  &  S
\end{pmatrix}
\begin{pmatrix}
I_k  &  B^{-1}C  \\
O_{k,r}  & I_r
\end{pmatrix}
\end{aligned}
\end{equation} 
and
\begin{equation}\label{eqgenpin}
\begin{aligned}
A^{-1}=\begin{pmatrix}
I_k  &  -B^{-1}C  \\
O_{k,r}  & I_r
\end{pmatrix}
\begin{pmatrix}
B^{-1}  &  O_{k,r} \\
O_{r,k}  &  S^{-1}
\end{pmatrix}
\begin{pmatrix}
I_k  &  O_{k,r}  \\
-DB^{-1}  & I_r
\end{pmatrix}
\end{aligned}.
\end{equation} 

We verify readily that $S^{-1}$ is the $(n-k)\times (n-k)$ trailing
(that is, southeastern) block of the inverse matrix $A^{-1}$, and so  
the Schur complement $S$ is nonsingular since the matrix $A$ is nonsingular.

Factorization (\ref{eqgenpin}) reduces the inversion of the matrix $A$ 
to the inversion of the leading block $B$ and its 
Schur complement $S$, and we can 
recursively reduce the task to the case of the leading blocks 
and Schur complements of decreasing sizes 
as long as the leading blocks are nonsingular.  
After sufficiently many  recursive steps 
of this
process of
block Gaussian elimination,
we only need to invert matrices  of
small sizes, and then we can stop the process   
 and  apply a selected black box
inversion algorithm. 

In $\lceil\log_2(n)\rceil$ recursive steps
all pivot blocks and 
all other
matrices involved into the 
resulting factorization
turn into scalars, 
all matrix 
multiplications and inversions turn into 
scalar multiplications and divisions,
and we arrive at a
{\em complete recursive factorization} of the matrix $A$.
If
$k=1$ at all recursive steps, then the complete
recursive factorization (\ref{eqgenpin})
defines GENP and can be applied to computing the inverse $A^{-1}$
or the solution ${\bf y}=A^{-1}{\bf b}$
to a linear system $A{\bf y}={\bf b}$. 

Actually, however, any complete recursive factorizations
turns into GENP up to the order in which we 
consider its steps.
This follows because  at most $n-1$  distinct
Schur complements $S=S(A^{(k)},A)$ 
for $k=1,\dots,n-1$
are involved in all recursive block factorization
processes for $n\times n$ matrices $A$,
and so we arrive at the same Schur complement in a fixed position
via GENP and via any other recursive block factorization (\ref{eqgenp}).
Hence we can interpret factorization step  (\ref{eqgenp})
as the block elimination of the
first $k$
columns of the matrix $A$, 
%representing the $k$ first variables,
which produces  the  matrix $S=S(A^{(k)},A)$.
If the dimensions 
$d_1,\dots,d_r$ and $\bar d_1,\dots,\bar d_{\bar r}$ of 
the pivot  blocks in 
two block elimination processes
 sum to the same integer $k$, that is, if 
$k=d_1+\dots+d_r=\bar d_1+\dots+\bar d_{\bar r}$,
then 
both processes produce the same Schur complement $S=S(A^{(k)},A)$.
The following results extend this observation.

% - - - - - - - - - - - - - - - - - - - - - - - - - - - - - - - - - - - - -

\begin{theorem}\label{thsch}
In the recursive block factorization process based on (\ref{eqgenp}),  
 every diagonal block of every block diagonal factor is either 
a leading block of the input matrix $A$ or the Schur complement $S(A^{(h)},A^{(k)})$
 for some integers $h$ and $k$ such that $0<h<k\le n$ and
%\begin{equation}\label{eqschurrec}
$S(A^{(h)},A^{(k)})=(S(A^{(h)},A))^{(h)}$.
%\end{equation}
\end{theorem}

% - - - - - - - - - - - - - - - - - - - - - - - - - - - - - - - - - - - - -

\begin{corollary}\label{corec}
The recursive block factorization process based on equation (\ref{eqgenp})  
can be completed by involving no singular
pivot blocks (and in particular no pivot elements vanish)
if and only if the input matrix $A$ is strongly nonsingular.
\end{corollary}
\begin{proof}
Combine Theorem \ref{thsch} with the equation 
$\det A=(\det B)\det S$,
implied by (\ref{eqgenp}).
\end{proof}
% - - - - - - - - - - - - - - - - - - - - - - - - - - - - - - - - - - - - -

The following theorem bounds the norms of all
 pivot blocks and their inverses and hence bounds
the condition numbers of the blocks, that is,
  precisely the quantities responsible for 
safe numerical performance of block Gaussian elimination 
and GENP. 
		
% - - - - - - - - - - - - - - - - - - - - - - - - - - - - - - - - - - - - -

\begin{theorem}\label{thnorms} (Cf. \cite[Theorem 5.1]{PQZ13}.)
Assume GENP
or block Gaussian elimination
applied to 
 an
$n\times n$
matrix $A$ and
write $N=||A||$ and $N_-=\max_{j=1}^n ||(A^{(j)})^{-1}||$,
and so $N_-N\ge ||A||~||A^{-1}||\ge 1$. 
Then the absolute values of all pivot elements of GENP 
and the norms of all pivot blocks of 
block Gaussian elimination
do not exceed $N_+=N+N_-N^2$,
while the absolute values of the reciprocals of these 
elements and the norms of the inverses of the blocks do not
exceed $N_-$.
\end{theorem}
\begin{proof}
Observe that the inverse $S^{-1}$ of the Schur complement $S$
in (\ref{eqgenp}) is the southeastern block of the inverse $A^{-1}$
and obtain
%So for the blocks $A_{00}$ of $S$ and $S^{-1}$ of $A^{-1}$ we surely have
$||B||\le N$, $||B^{-1}||\le N_-$, and 
$||S^{-1}||\le ||A^{-1}||\le N_-$. Moreover
$||S||\le N+N_-N^2$, due to (\ref{eqgenp}). Now the claimed bound follows 
from Theorem \ref{thsch}.
\end{proof}
\noindent
 
% - - - - - - - - - - - - - - - - - - - - - - - - - - - - - - - - - - - - -

\begin{remark}\label{reir}
By virtue of Theorem \ref{thnorms} 
the norms of the inverses of
all pivot blocks involved into a complete
(and hence also into any incomplete) recursive factorization
of a strongly nonsingular matrix $A$
are at most $N_-$.
We have a reasonable upper bound on $N_-$ if
the matrix $A$ is strongly well-con\-di\-tioned as well.
Then
in view of Theorem \ref{thpert}
the inversion of all pivot blocks is numerically
safe, and we say that {\em GENP is locally safe} for the matrix $A$.
%Locally safe recursive factorization involves 
%either divisions by small pivot entries 
%(and avoiding them is the purpose of pivoting) 
% nor inversions of ill-con\-di\-tioned pivot blocks. 
%for GEPP (cf. \cite[page 119]{GL96} and \cite[Theorem 3.4.12]{S98}).
\end{remark}

%- --- ---------------------- - - - - - - - - - - - - - - - - - - - - - - -

\begin{remark}\label{renrm}
In the recursive factorizations above 
only the factors of the leading blocks  
and the Schur complements
can contribute to the magnification
of any input perturbation. Namely at most $\lceil\log_2(n)\rceil$
such factors can contribute to the norm of each 
of the output triangular or block triangular
factors $L$ and $U$. This implies the moderately large 
worst case upper bound $(N_+N_-)^{\log_2(n)}$
on their norms, which  
 is overly pessimistic
according to our tests.
%, which can be compared favorably to the sharp upper bound $2^{n-1}$
%on the growth factor
%for GEPP (cf. \cite[page 119]{GL96} and \cite[Theorem 3.4.12]{S98}).
\end{remark}

% - - - - - - - - - - - - - - - - - - - - - - - - - - - - - - - - - - - - -

\begin{remark}\label{rerect}
Our study in this and the next two sections can be extended readily to 
the cases of GENP and block Gaussian 
elimination applied to rectangular and possibly rank deficient matrices 
and to under- and over-determined and possibly rank deficient 
linear systems of equations.
Recursive factorization and elimination 
can be completed and are numerically safe when they are applied to any  
strongly nonsingular and  strongly well-conditioned leading block 
of the input matrix, in particular 
to the input matrix 
itself if it is strongly nonsingular and  strongly well-conditioned.
\end{remark}

%------------------------------------------------------------------------------

\section{Singular values of the  matrix products (deterministic estimates)
and GENP and block Gaussian elimination with preprocessing}\label{smrc}

%------------------------------------------------------------------------------

{\em Preprocessing} $A\rightarrow FAH$ 
for a pair of nonsingular matrices $F$ and $H$, 
one of which can be the identity matrix $I$,
reduces
the inversion of a matrix $A$ to
the inversion of a the product $FAH$,
and similarly for the solution of a linear system of equations.  
\begin{fact}\label{faprepr}
%(Preprocessing.)
Assume three nonsingular matrices $F$, $A$, and $H$
 and a vector ${\bf b}$.
Then 
%\begin{equation}\label{eqprepr}
$A^{-1}=H(AH)^{-1}$, $A^{-1}=(FA)^{-1}F$,
$A^{-1}=H(FAH)^{-1}F$. Moreover,  if $A{\bf x}={\bf b}$, then $AH{\bf y}={\bf b}$, $FA{\bf x}=F{\bf b}$, and 
$FAH{\bf y}=F{\bf b}$, ${\bf x}=H{\bf y}$. 
%\end{equation}
\end{fact}

 Remark \ref{reir} motivates the choice of the multipliers $F$ 
and $H$ for which the 
matrix $FAH$ is strongly nonsingular and
 strongly well-con\-di\-tioned. 
This is likely to occur already if one of the multipliers
$F$ and $H$
 is the identity matrix 
 and another one is a Gaussian random matrix. 
The studies of 
pre-multiplication by $F$ and post-multiplication by $H$
are similar, and so we only prove the latter claim
in the case of post-multiplication.
We complete our proof in Section \ref{ssgnp}. It
 involves the norms of the inverses of the matrices 
 $(AH)_{k,k}=A_{k,n}H_{n,k}$ for $k=1,\dots,r$,
 which we  estimate in this section  
assuming  
nonrandom multipliers
$H$. 
%(Clearly, $||(FA)_{k,k}||\le ||F_{k,m}||~||A_{m,k}||)$
%and $||(AH)_{k,k}||\le ||A_{k,n}||~||H_{n,k}||)$.)
%Empirically,  the support for GENP is not weaken
%where we use
%preprocessing with
%random circulant rather than Gaussian multipliers.
We begin with  two simple lemmas.
% and then
%analyze the impact that pre- and post-multiplication
%of a matrix  makes on
%its singular values, that is, we 
%estimate the singular values of the %matrix products in terms of the singular values of the multipliers
%and the input matrix.

%he first of them is obvious,
%the second easily follows  
%from minimax property (\ref{eqminmax}).
%------------------------------------------------------------------------------

\begin{lemma}\label{lepr2} 
If $S$ and $T$ 
are square orthogonal matrices, then 
$\sigma_{j}(SA)=\sigma_j(AT)=\sigma_j(A)$ for all $j$.
\end{lemma}

%------------------------------------------------------------------------------

\begin{lemma}\label{lepr1} 
Suppose $\Sigma=\diag(\sigma_i)_{i=1}^{n}$, $\sigma_1\ge \sigma_2\ge \cdots \ge \sigma_n$,
 and  $H\in \mathbb R^{n\times r}$.
Then 
$$\sigma_{j}(\Sigma H)\ge\sigma_{j}(H)\sigma_n {\rm for~all}~j.$$
If also $\sigma_n>0$, then 
 $$\rank (\Sigma H)=\rank (H).$$
\end{lemma}

%------------------------------------------------------------------------------

We also need the following basic results (cf. \cite[Corollary 8.6.3]{GL13}).
 
\begin{theorem}\label{thcndsb} 
If $A_0$ is a 
submatrix of a 
matrix $A$, 
then
$\sigma_{j} (A)\ge \sigma_{j} (A_0)$ for all $j$.
\end{theorem} 

%--------------------------------------------------------------------------
%------------------------------------------------------------------------------
 
%\begin{fact}\label{faintpr} (Cf. \cite[Corollary 8.6.3]{GL13}.)  
%Suppose $r+l\le n\le m$, $l\ge 0$,
% $1\le k\le r$,
 %$A\in \mathbb R^{m\times n}$,
%and
%$A_{m,r}$ is the leftmost $m\times r$ block of the matrix $A$. 
%Then $\sigma_{k} (A_{m,r})\ge \sigma_{k+l} (A_{m,r+l})$.
%\end{fact}

%------------------------------------------------------------------------------

\begin{theorem}\label{thintpr}  
Suppose $r+l\le n\le m$, $l\ge 0$,
 $A\in \mathbb R^{m\times n}$,
$\rank(A_{m,r})=r$ and $\rank(A_{m,r+l})=r+l$.
%for the matrices $A_{m,r}$  and $A_{m,r+l}$.
Then $||A_{m,r}^+||\le ||A_{m,r+l}^+||$.
\end{theorem}

%------------------------------------------------------------------------------

The following theorem will enable us to
estimate the norm $||(AH)^+||$.

\begin{theorem}\label{1}
%Suppose
Suppose $A\in \mathbb R^{n\times n}$,
 $H\in\mathbb R^{n\times r}$,
$\rank (A)=n\ge r$,  
 $A=S_A\Sigma_AT^T_A$ is SVD (cf. (\ref{eqsvd})), 
 and
$\widehat H=T_A^TH$.
 %$\widehat H_{s,r}=(I_{s}~|~O)\widehat H$ for all $s\le n$, and
 %$\widehat F_{r,q}=\widehat F(I_{q}~|~O)$ for all $q\le m$.
 Then 
\begin{equation}\label{eqah}
\sigma_j(AH)\ge \sigma_l(A) ~\sigma_j(\widehat H_{l,r})~{\rm for~all}~l\le n~{\rm and~all}~j.
\end{equation}
\end{theorem}
%The estimates can be substantially strengthened in the case of rectangular 
%matrices $G$ where $r(G)\ll m$ and $H$ where $r(H)\ll n$
%(cf. Remark \ref{rerect}).

%------------------------------------------------------------------------------

\begin{proof}
Note that
 $AH=S_A\Sigma_AT_A^TH$, and so  
$\sigma_j(AH)=\sigma_j(\Sigma_AT_A^TH)=
\sigma_j(\Sigma_A\widehat H)$ for all $j$ 
by virtue of Lemma \ref{lepr2}, because 
$S_A$ is a square orthogonal matrix. 
Moreover it follows from Theorem \ref{thcndsb} that
$\sigma_j(\Sigma_A\widehat H)\ge 
\sigma_j(\Sigma_{l,A}\widehat H_{l,r})$
for all $l\le n$. Combine this bound with the latter equations
and apply Lemma \ref{lepr1}. 
\end{proof}

%------------------------------------------------------------------------------

\begin{corollary}\label{cor1}
Keep the assumptions 
%and definitions 
of Theorem \ref{1}. Then 

(i)  $\sigma_{r}(AH)\ge 
\sigma_{\rho}(A)\sigma_{r}(\widehat H_{n,r})=
\sigma_{r}(\widehat H_{n,r})/||A^+||$,

(ii) $||(AH)^+||\le ||A^+||~||\widehat H_{n,r}^+||$
if $\rank(AH)=\rank(\widehat H_{n,r})=r$.
\end{corollary}

%------------------------------------------------------------------------------

\begin{proof}
Substitute $j=r$ and $l=n$
into bound (\ref{eqah}), recall
(\ref{eqnrm+}), and
obtain part (i).
If $\rank(AH)=\rank(\widehat H_{l,r})=r$,
then apply (\ref{eqnrm+})
to obtain that 
$\sigma_{r}(AH)=1/||(AH)^+||$ and 
$\sigma_{r}(\widehat H_{l,r})=1/||\widehat H_{l,r}^+||$. 
Substitute these equations 
into part (i) and obtain part (ii).
\end{proof}

%------------------------------------------------------------------------------

Let us  extend the estimates of Theorem \ref{1}  to the leading
 blocks of a matrix product.

%------------------------------------------------------------------------------

\begin{corollary}\label{cogh}
Keep the assumptions 
%and definitions 
of Theorem \ref{1} and also suppose that
 the matrices  $(AH)_{k,k}$ and $\widehat H_{n,k}$ 
have full rank $k$ for
a positive integer
$k\le n$. 
% $M_{q,s}=(I_q~|~O_{q,k-q})M\begin{pmatrix}I_s\\O_{l-s,s}\end{pmatrix}$
 %for $q=1,\dots,m$ and $s=1,\dots,n$.
%$\rank (A\begin{pmatrix}I_j   \\   O_{n-j,j}\end{pmatrix})=j$,
%$\rank ((I_k~|~O_{k,m-k})A)=k$ for   
Then  
 $$||(AH)_{k,k}^+||\le ||\widehat H_{n,k}^+||~||A_{k,n}^+||\le
||\widehat H_{n,k}^+||~||A^+||.$$ 
\end{corollary} 

%------------------------------------------------------------------------------

\begin{proof}
% Part (i) follows from Corollary \ref{corrk} and  bound (\ref{eq1})
%because $(FAH)_{q,s}=F_{q,m}AH_{n,s}$.
Note that 
$(AH)_{k,k}=A_{k,n}H_{n,k}$
and that the matrix $A_{k,n}$ has full rank.
Apply Corollary \ref{cor1}  for $A$ and $H$ replaced by 
$A_{k,n}$ and $H_{n,k}$, respectively,
and obtain that $||(AH)_{k,k}^+||\le ||\widehat H_{k,n}^+||~||A_{n,k}^+||$.
Combine (\ref{eqnrm+}) and Theorem \ref{thintpr} and deduce that 
$||A_{n,k}^+||\le||A^+||$. Combine  the two latter 
inequalities 
to complete the proof of part (i). 
Similarly prove part (ii).
\end{proof}

%------------------------------------------------------------------------------

Fact \ref{faprepr}, 
Corollary \ref{corec} and Theorem \ref{thnorms} together imply the
following result.

%------------------------------------------------------------------------------

\begin{corollary}\label{colocsf}
 Suppose that
$A\in \mathbb R^{n\times n}$,
$H\in\mathbb R^{n\times r}$,
 $r\le n=\rank (A)$,  and 
 the matrices $(AH)_{k,k}$ 
 %for $m\ge n$
are   
strongly nonsingular and
   strongly well-con\-di\-tioned
for $k=1,\dots,r$. Then
 GENP  and 
block Gaussian elimination
are locally safe for the
matrix product $AH$
 (see Remark \ref{reir} on the concept ``locally safe").
\end{corollary}

%------------------------------------------------------------------------------

\section{Benefits of using Gaussian multipliers for GENP and block Gaussian elimination
%Approximate bases for singular spaces,  and the
%computation of numerical rank
}\label{sapsr1}

%------------------------------------------------------------------------------

 In  Section
\ref{srrm} we recall the norm and condition estimates for  
Gaussian matrices and deduce that these matrices are
strongly nonsingular with probability 1 and
are expected to be strongly well-conditioned.
In Section \ref{ssgnp} we prove
that the pair $(H,A)$
for a nonsingular and well conditioned  matrix $A$ and a
Gaussian matrix $H$ is expected to
satisfy the assumptions of Corollary \ref{colocsf},  
implying that the application of 
GENP  and block Gaussian elimination 
to the product $AH$ is  numerically safe.
%In Section \ref{srnd} we comment on using non-Gaussian random multipliers. 

%------------------------------------------------------------------------------

\begin{remark}\label{remu0}
The above results do not hold 
if the mean greatly exceeds standard deviation 
of the i.i.d. entries of a
multipliers $H$. Its power for
 achieving numerically safe GENP
and block Gaussian elimination
is usually lost in this case.
Indeed assume 
 a mean $\mu$ and a standard deviation
$\sigma$ such that
$\mu\gg \sigma$ (already   
$\mu>10 \sigma\log (n)$ is suficient).
In this case the matrix $H$ is expected to be 
closely approximated by the rank-1 matrix $\mu {\bf e}{\bf e}^T$
where ${\bf e}^T=(1,1,\dots,1)$.
\end{remark}

%------------------------------------------------------------------------------

\subsection{A Gaussian 
% general, Toeplitz and circulant 
matrix, its rank, norm and condition estimates}\label{srrm}

%------------------------------------------------------------------------------

\begin{definition}\label{defrndm}
A matrix is said to be {\em standard Gaussian random} 
(hereafter we say just
{\em Gaussian}) if it is filled with i.i.d.
Gaussian random
variables  having mean $0$ and variance $1$. 
%$\mathcal G_{m,n}$ is  Gaussian matrix of size $m\times n$.  
\end{definition}

%------------------------------------------------------------------------------

\begin{theorem}\label{thdgr1}
A Gaussian matrix $G$ is strongly nonsingular 
with probability 1.
\end{theorem}
\begin{proof}
Assume that the $j\times j$ leading submatrix $G^{(j)}$ of a $k\times l$
Gaussian matrix $G$ is singular for some positive integer $j\le h=\min\{k,l\}$,
that is, $\det(G^{(j)})=0$. 
Since $\det(G^{(j)})$ is a polynomial in the entries of the Gaussian matrix 
 $G^{(j)}$, such matrices form 
%$\begin{pmatrix} m \\n \end{pmatrix}$
an algebraic variety of a lower dimension  in the linear space 
$\mathbb R^{j^2}$. 
($V$ is an algebraic variety of a dimension $d\le N$
in the space $\mathbb  R^{N}$ if it is defined by $N-d$
polynomial equations and cannot be defined by fewer equations.)
Clearly, Lebesgue (uniform) and
Gaussian measures of such a variety equal 0,  being absolutely continuous
with respect to one another. 
Hence these measures of the union of $h$ such matrices
%over $j=1,\dots, h$ 
are also 0.
\end{proof}
%\begin{definition}\label{defcdfm}
%Hereafter we simplify the notation by writing $F_{M}(y)=F_{||M||}(y)$ for a matrix $M$.
% and an integer $l=\min\{m,n\}$. 
%\end{definition}

%------------------------------------------------------------------------------

\begin{theorem}\label{thdgr}
Assume a nonsingular $n\times n$ matrix $A$ and an $n\times k$  
 Gaussian matrix $H_{n,k}$.
Then the product $A_{k,n}H_{n,k}$ is nonsingular with probability 1.
\end{theorem}
\begin{proof}
$\det(A_{k,n}H_{n,k})$ is a polynomials in the entries of the Gaussian matrix $H_{n,k}$.
Such a polynomial vanishes with probability 0 unless it vanishes identically in  $H_{n,k}$,
but  the matrix 
$A_{k,n}A_{k,n}^T$ is positive definite, and so
$\det(A_{k,n}H_{n,k})>0$ for $H_{n,k}=A_{k,n}^T$.
\end{proof}

%------------------------------------------------------------------------------

\begin{definition}\label{defnu}
 $\nu_{j,m,n}$ denotes the random variables
$\sigma_j(G)$ for a Gaussian $m\times n$ matrix $G$
and all $j$, while
$\nu_{m,n}$, $\nu_{F,m,n}$,  $\nu_{m,n}^+$, and $\kappa_{m,n}$ denote the 
 random variables
 $||G||$,  $||G||_F$, 
 $||G^+||$, and $\kappa(G)=||G||~||G^+||$, respectively.
\end{definition}

Note that $\nu_{j,n,m}=\nu_{j,m,n}$, $\nu_{n,m}=\nu_{m,n}$,
$\nu_{n,m}^+=\nu_{m,n}^+$, and $\kappa_{n,m}=\kappa_{m,n}$.

\begin{theorem}\label{thsignorm}
(Cf. \cite[Theorem II.7]{DS01}.)
Suppose 
%$G$ is a  Gaussian $m\times n$ matrix,
$h=\max\{m,n\}$, $t\ge 0$,  and
$z\ge 2\sqrt {h}$. 
Then  ${\rm Probability}~\{\nu_{m,n}>z\}\le
\exp(-(z-2\sqrt {h})^2/2\}$ and 
${\rm Probability}~\{\nu_{m,n}>t+\sqrt m+\sqrt n\}\le
\exp(-t^2/2)$.
\end{theorem}
%------------------------------------------------------------------------------

\begin{theorem}\label{thsiguna} 
(Cf. \cite[Proof of Lemma 4.1]{CD05}.)
Suppose 
$m\ge n\ge 2$, and $x>0$ and write $\Gamma(x)=
\int_0^{\infty}\exp(-t)t^{x-1}dt$ 
and $\zeta(t)=
%\frac{\sqrt{2m}}{\Gamma(m/2)}(t\sqrt{m/2})^{m-1}\exp(-mt^2/2)=
t^{m-1}m^{m/2}2^{(2-m)/2}\exp(-mt^2/2)/\Gamma(m/2)$. 
Then 
 ${\rm Probability}~\{\nu_{m,n}^+\ge m/x^2\}<\frac{x^{m-n+1}}{\Gamma(m-n+2)}$.
\end{theorem}

%------------------------------------------------------------------------------

The following condition estimates
 from \cite[Theorem 4.5]{CD05}
%although we do not use them.
are quite tight for large values $x$, but for $n\ge 2$ even tighter
estimates (although more involved)
 can be found in \cite{ES05}. (See \cite{D88} and  \cite{E88}
on the early study.)

%------------------------------------------------------------------------------

\begin{theorem}\label{thmsiguna}   
%$A\in \mathcal G^{m\times n}$. 
If $m\ge n\ge 2$, then 
$${\rm Probability}~\{\kappa_{m,n}m/(m-n+1)>x\}\le \frac{1}{2\pi}(6.414/x)^{m-n+1}$$
for $x\ge m-n+1$, while $\kappa_{m,1}=1$
with probability 1.
\end{theorem}

%------------------------------------------------------------------------------

%\begin{lemma}\label{lebas}  
%Suppose $y$ is a positive number, ${\bf t}\in \mathbb R^{n\times 1}$ 
%is a unit vector, $||{\bf t}||=1$,
% $A\in \mathcal G_{\mu,\sigma}^{n\times n}$
%a Gaussian random matrix independent of the matrix $M$ and having a mean $\mu$ and a variance $\sigma^2$, 
%and therefore is nonsingular with probability $1$, 
% $Q$ is a orthogonal matrix, 
%$B=QA=({\bf b}_1~|~\dots~|~{\bf b}_n)$,
% and ${\bf t}^T{\bf b}_i=0$ for $i=2,\dots,n$. 
%Then 
%$${\rm Probability}\{||(QA)^{-1}{\bf e}_1||>y\}
%\le \max_{{\bf b}_2,\dots,{\bf b}_n}{\rm Probability}_{{\bf b}_1}\{|{\bf t}^T{\bf b}_1|<1/y\}.$$
%\end{lemma}
 
%\begin{proof}
%See \cite[the proof of Lemma 3.2]{SST06}.
%\end{proof}

%------------------------------------------------------------------------------

\begin{corollary}\label{cogfrwc} 
 A Gaussian 
%general, Toeplitz or circulant
matrix  is expected to be  strongly well-con\-di\-tioned.
\end{corollary}

%--------------------------------------------------------------------------------

\subsection{Supporting  GENP with Gaussian multipliers
%Approximate bases for singular spaces,  and the
%computation of numerical rank
}\label{ssgnp}

%------------------------------------------------------------------------------

The main result of this section  is
Corollary \ref{cogmforalg}, which supports application of 
GENP  and block Gaussian elimination 
to the product $AH$ of a nonsingular matrix $A$ and  a
Gaussian matrix $H$.

%------------------------------------------------------------------------------

We need the following simple basic lemma.

%------------------------------------------------------------------------------

\begin{lemma}\label{lepr3}
%\cite[Proposition 2.2]{SST06}.
Suppose $H$ is a Gaussian matrix, 
$S$ and $T$ are orthogonal matrices, $H\in \mathbb R^{m\times n}$,
$S\in \mathbb R^{k\times m}$, and $T\in \mathbb R^{n\times k}$
for some $k$, $m$, and $n$.
Then $SH$ and $HT$ are Gaussian matrices.
\end{lemma}

%------------------------------------------------------------------------------

\begin{corollary}\label{cor10}
Suppose $A$ is a nonsingular $n\times n$ matrix,
$H$ is an $n\times k$ Gaussian matrix,
for $0<k\le n$,
and $\nu_{g,h}$ and $\nu_{g,h}^+$ are the random values 
of Definition \ref{defnu}. 
Then 

(i) the matrix $(AH)_{k,k})$ is nonsingular 
with probability $1$,

(ii) $||(AH)_{k,k}||\le \nu_{n,k}||A_{k,n}||\le \nu_{n,k}||A||$, and

 (iii) $||(AH)_{k,k}^{+}||\leq \nu_{n,k}^+||A^{+}||$.
%\end{equation}
\end{corollary}
%The estimates can be substantially strengthened in the case of rectangular 
%matrices $G$ where $r(G)\ll m$ and $H$ where $r(H)\ll n$
%(cf. Remark \ref{rerect}).

%------------------------------------------------------------------------------

\begin{proof}
Part (i) restates Theorem \ref{thdgr}.
Part (ii) follows because $(AH)_{k,k}=A_{k,n}H_{n,k}$,
 $H_{n,k}$ is a Gaussian  matrix,
 and $||A_{k,n}||\le ||A||$. 
Part (iii) follows from Corollary \ref{cogh} 
because 
 $\widehat H_{n,k}$ is a Gaussian  matrix
by virtue of Lemma \ref{lepr3}. 
\end{proof}

%------------------------------------------------------------------------------
%------------------------------------------------------------------------------

\begin{corollary}\label{cogmforalg}
Suppose that the $n\times n$ matrix $A$
is nonsingular and  well-con\-di\-tioned.
Then the  choice of 
Gaussian multiplier 
$H$  is expected 
to satisfy the assumptions of Corollary \ref{colocsf}.
\end{corollary}
\begin{proof}
Recall that $(AH)_{k,k}=A_{k,n}H_{n,k}$
and hence $||(AH)_{k,k}||=||A_{k,n}||\nu_{n,k}$.
Then
combine Theorems 
\ref{thdgr},
\ref{thsignorm}, and
\ref{thsiguna} 
and Corollary
 \ref{cor10}.
\end{proof}

%------------------------------------------------------------------------------

\section{Random structured multipliers for GENP and block Gaussian elimination}\label{srnd}

% - - - - - - - - - - - - - - - - - - - - - - - - - - - - - - - - - - - - -

This subsection involves complex  matrices. A complex matrix $M$ is  unitary if
 $M^HM=I$ or $MM^H=I$ where $M^H$ denotes its Hermitian transpose,
so that $M^H=M^T$ for a real matrix $M$. 

Hereafter 
 $\omega=\omega_n=\exp(\frac{2\pi}{n} \sqrt {-1})$ 
denotes an $n$th primitive root of unity,
$\Omega=(\omega^{ij})_{i,j=0}^{n-1}$ 
is the matrix of the discrete Fourier transform 
 at $n$ points
(we use the acronym {\em DFT}),
and $\Omega^{-1}=\frac{1}{n}\Omega^{H}$. 

An $n\times n$
circulant matrix $C=(c_{i-j\mod n})_{i,j=0}^{n-1}$ 
is defined by its first column ${\bf c}=(c_i)_{i=0}^{n-1}$.

\begin{example}\label{exc1} {\em  Generation of random real circulant
  matrices.}
Generate the vector ${\bf c}$ of $n$ i.i.d. random real variables 
in the range $[-1,1]$ under the uniform probability distribution
on this range.
Define an $n\times n$ circulant matrix $C$ with the first column ${\bf c}$.
\end{example}

The following theorem links 
the matrices $\Omega$ and $\Omega^{-1}$ to the  class of circulant matrices.

\begin{theorem}\label{thcpw} (Cf. \cite{CPW74}.)
 Let
$C$ denote a circulant $n\times n$ matrix defined by its
first column ${\bf c}$ and write 
${\bf u}=(u_i)_{i=1}^n=\Omega {\bf c}$.
Then
$C=\Omega^{-1}\diag(u_j)_{j=1}^n\Omega$. Furthermore 
$C^{-1}=\Omega^{-1}\diag(1/u_j)_{j=1}^n\Omega$ if 
the matrix $C$ is nonsingular.
\end{theorem}
By using FFT, one can multiply the matrices $\Omega$
and $\Omega^H=\Omega^{-1}$ by a vector by using 
$O(n\log (n))$ flops for any $n$ (cf., e.g., \cite[page 29]{p01}), and 
Theorem \ref{thcpw} extends this complexity bound to 
multiplication  of an
$n\times n$ 
circulant matrix and its inverses by a vector.

We need $2n^3-n^2$ flops in order to compute the product $AH$  of
the pair of $n\times n$ matrices $A$  
and  $H$.
If, however, $H$ is a   
 circulant matrix, then 
we can compute $AH$ by using order of $n^2\log (n)$ flops.
For a Toeplitz-like matrix $A$
defined by its displacement generator of bounded length $l$,
we use  $O(ln\log(n))$ flops in order to compute a
displacement generator of length $l$ for the matrix $AH$.
(See \cite{p01} for the definition of displacement generators.)
In the case of Toeplitz matrices we have $l\le 2$ and use
 $O(n\log(n))$ flops.
This motivates using Gaussian circulant multipliers $H$,
that is, circulant matrices $H$ whose first column vector is 
Gaussian. It has been proved in \cite{PSZa}
that such matrices are expected to be well-conditioned,
which is required for any multiplicative preconditioner.
%based on the following simple corollary of Theorem \ref{thcpw}
%and Lemma \ref{lepr3}.

%\begin{corollary}\label{cocrclcnd} (Cf. \cite{PSZa}.)
%Assume 
% a nonsingular circulant matrix $C$ with the first column 
%{\bf c}$ and let ${\bf u}=\Omega {\bf c}$,
%as in Theorem \ref{thcpw}.
%Then 
%$$||C||=||\diag ({\bf u})||=\max_{j=1}^n |u_j|,~
%||C^{-1}||= ||(\diag ({\bf u}))^{-1}||=1/\min_{j=1}^n |u_j|,$$ 
%and so
% $$\kappa(C)=||C||~||C^{-1}||=\max_{i,j=1}^n |u_i/u_j|.$$
%Moreover if ${\bf c}$ is a Gaussian vector, then so is the vector ${\bf u}/\sqrt n$.
%\end{corollary}

We can define a unitary circulant matrix
by its first column vector
${\bf c}=\Omega (\exp(r_i\sqrt {-1}))_{i=0}^{n-1}$
for any set of real values $r_0,\dots,r_{n-1}$.

\begin{example}\label{exunc} {\em Generation of  random unitary circulant 
 matrices.}

(i) Generate a vector ${\bf u}=(u_j)_{j=1}^n$ where
 $u_j=\exp(2\pi \phi_j \sqrt {-1})$ (and so
$|u_j|=1$ for all $i$) and  where $\phi_1,\dots,\phi_n$
 are $n$ independent random real variables, e.g.,
Gaussian  variables or the variables uniformly distributed in the range $[0,1)$. 

(ii) Compute the vector  
${\bf c}=\Omega^{-1}{\bf u}$,
where  $\Omega$
denotes the $n\times n$ DFT  matrix.
Output 
the unitary circulant matrix $C$ 
defined by its first column ${\bf c}$.
\end{example}

Our proof that Gaussian multipliers enforce strong nonsingularity 
of a nonsingular matrix with probability 1 (see Theorem \ref{thdgr}) 
has been non-trivially extended in \cite{PZa} to the case of 
Gaussian circulant multipliers.
Furthermore 
strong nonsingularity  holds
with probability close to 1 if we 
fill the first column of a multiplier $F$ or $H$ 
with i.i.d. random variables defined
under the uniform probability distribution
over a sufficiently large finite set
(see Appendix \ref{srsnrm} and  \cite{PSZa}).

In our tests with random input matrices, 
Gaussian circulant and general Gaussian multipliers have shown 
the same power  
of supporting numerically safe GENP 
(see Section \ref{sexgeneral}), but we cannot extend
our basic Lemma \ref{lepr3} and
our Corollary \ref{cogmforalg}
to the  case of circulant matrices.
Moreover our Theorem \ref{thcgenp} and
Remark \ref{remext}  below
show that, for a  specific narrow class of input matrices $A$, 
GENP with these multipliers is expected to fail numerically. 

\begin{theorem}\label{thcgenp}
Assume a large integer $n$ and the $n\times n$ DFT matrix $\Omega$,
which is unitary up to scaling by $1/\sqrt n$. 

(i) Then application of GENP to this matrix fails numerically and
 
(ii) a Gaussian circulant $n\times n$ multiplier $C=\Omega^{-1}D\Omega$
with Gaussian diagonal matrix $D=\diag(g_j)_{j=1}^n$
(having i.i.d. Gaussian diagonal entries $g_1,\dots,g_n$)
is not expected to fix this problem. 
\end{theorem}
\begin{proof}
 (i) Subtract  the 
first row of the block 
$\Omega_{2,2}$ of the matrix $\Omega$ 
and  the resulting vector
from its second row. Obtain the vector
 $(0,\omega-1)$ with the norm $|\omega-1|=2\sin (\pi/n)$.
Assume that $n$ is large and then observe that 
$2\sin (\pi/n)\approx 2\pi/n$ and that
 the  variable $2\pi/n$ is expected to be small, implying that
 $\nrank(\Omega_{2,2})=1$ because $||\Omega_{2,2}||\ge \sqrt 2$
for large $n$. 

(ii) Note that $\Omega C=D \Omega$.
The Gaussian variable $g_1$
vanishes with probability 0, 
and so we can assume that $g_1\neq 0$.
 Multiply  the 
first row of the block 
$(D\Omega)_{2,2}$ of the matrix $D\Omega$ by $g_2/g_1$
and subtract the resulting vector
from the second row. Obtain the vector
 $(0,(\omega-1)g_2)$ with the norm
 $|(\omega-1)g_2|=2|g_2\sin (\pi/n)|$
Assume that $n$ is large and then observe that 
$|(\omega-1)g_2|\approx 2|g_2|\pi/n$
and that
 the  variable $2|g_2|\pi/n$ is expected to be small. 
Hence
 $\nrank((\Omega C)_{2,2})=\nrank((D\Omega)_{2,2})$ is expected to equal 1
 because 
$||\Omega C)_{2,2}||\le ||\Omega_{2,2}||\max\{g_1,g_2\}$,
$||\Omega_{2,2}||\ge \sqrt 2$
 and the random variable  $\max\{|g_1|,|g_2|\}$
is not expected to be close to 0.
\end{proof}

\begin{remark}\label{remext}
The same argument shows that 
Gaussian circulant multipliers $C$ are not expected to support GENP
for a bit larger class of matrices, e.g., for $A=M\Omega$ where $M=\diag(D_i)_{i=1}^k$,
 $D_1=\diag(d_1,d_2)$,  and $d_1$  and $d_2$ are two positive constants
and the input size $n\times n$ is large 
as well as where the matrix  $M$ is strongly diagonally dominant. 
The reader is challenged to 
find out whether   
GENP with a Gaussian circulant preprocessor is 
expected to fail numerically for other  classes of input matrices,
in particular for any
  subclass of the classes of Toeplitz or
Toeplitz-like matrices (cf. \cite{p01} and \cite{P15} on these classes).
Another challenge is to choose a distinct random structured preprocessor
for which the above problem is avoided. E.g., consider
the product $\prod_{i=1}^h C_i$ where $h$ is 
 a small integer exceeding 1,
$C_{2j}$
are circulant matrices and $C_{2j-1}$ are skew-circulant 
(see the definition in \cite{p01}).
 Toward the same goal
 we can apply simultaneously random structured
 pre- and post-multipliers $F$ and $H$,
defined by some i.i.d. random parameters, or the pairs of 
PRMB multipliers of \cite{BBD12}. In the case of Toeplitz or Toeplitz-like
input matrices $A$, the multiplications $FA$ and $AH$ are much 
less costly if the multipliers $F$ and $H$ are
circulant matrices, skew-circulant matrices, or the products of 
such matrices.
\end{remark}

%------------------------------------------------------------------------------

\section{Low-rank approximation
%Approximate bases for singular spaces,  and the
%computation of numerical rank
}\label{sapsr}

%------------------------------------------------------------------------------
%------------------------------------------------------------------------------
%------------------------------------------------------------------------------

Suppose we seek a rank-$r$ approximation of a matrix $ A$ 
that has a small numerical rank $r$.
One can solve this problem by computing 
 SVD of the 
 matrix $A$
 or, at a lower cost, by computing
 its rank-revealing 
 factorization \cite{GE96}, \cite{HP92}, \cite{P00a},
but using
random matrix multipliers instead
has some benefits \cite{HMT11}.
In this section we study the latter randomized approach. In
 its first subsection we recall some relevant definitions and 
auxiliary results. 

%------------------------------------------------------------------------------

\subsection{Truncation of SVD. Leading and trailing singular spaces}\label{sosvdi1}

%------------------------------------------------------------------------------

Truncate the square orthogonal matrices $S_A$ and $T_A$
and the square diagonal matrix $\Sigma_{A}$ of the SVD of (\ref{eqsvd}),
write $S_{\rho,A}=(S_A)_{m,\rho}$,
 $T_{\rho,A}=(T_A)_{n,\rho}$, and  $\Sigma_{\rho,A}=(\Sigma_{A})_{\rho,\rho}=\diag(\sigma_j)_{j=1}^{\rho}$,
and obtain {\em thin SVD}
\begin{equation}\label{eqthn}
A=S_{\rho,A}\Sigma_{\rho,A}T_{\rho,A}^T,~\rho=\rank (A).
\end{equation}
Now for every integer $r$ in the range $1\le r\le\rho=\rank(A)$,  
write $\Sigma_{\rho,A}=\diag(\Sigma_{r,A},\bar\Sigma_{A,r})$ and
partition
the matrices $S_{\rho,A}$ and $T_{\rho,A}$ into block columns,
 $S_{\rho,A}=(S_{r,A}~|~\bar S_{A,r})$, and 
 $T_{\rho,A}=(T_{r,A}~|~\bar T_{A,r})$
where $\Sigma_{r,A}=(\Sigma_{A})_{r,r}=\diag(\sigma_j)_{j=1}^r$,
$S_{r,A}=(S_A)_{m,r}$,
and $T_{r,A}=(T_A)_{n,r}$.
Then  
 partition the thin SVD as follows,
\begin{equation}\label{eqldtr}
A_r=S_{r,A}\Sigma_{r,A}T_{r,A}^T,~\bar A_r=\bar S_{A,r}\bar \Sigma_{A,r}\bar T_{A,r}^T,~
A=A_r+\bar A_r~{\rm for}~ 1\le r\le\rho=\rank(A),
\end{equation} 
and call the above decomposition 
%of the matrix $A_r$
the $r$-{\em truncation of  thin SVD} (\ref{eqthn}).
Note that $\bar A_{\rho}$ is an empty matrix
and recall that
\begin{equation}\label{eqsvdpert} 
||A-A_r||= \sigma_{r+1}(A).
\end{equation}
Let
$\mathbb S_{r,A}$ and  $\mathbb T_{r,A}$ denote
the ranges (that is, the column spans)
 of the matrices $S_{r,A}$ and $T_{A,r}$, respectively.
If $\sigma_r>\sigma_{r+1}$, 
then 
$\mathbb S_{r,A}$ and $\mathbb T_{r,A}$ are
the left and right {\em leading singular spaces}, respectively,
    associated with the $r$ largest singular values of the matrix $A$.
%whereas their orthogonal complements $\mathbb S_{A,m-r}=\mathcal R(S_{A,m-r})$ 
%and $\mathbb T_{A,n-r}=\mathcal R(T_{A,n-r})$ 
%of these singular spaces  
%are the left and right {\em trailing singular spaces}, respectively,
%associated with the other singular values of $A$. 
%The pairs of subscripts $\{r,A\}$ versus $\{A,m-r\}$ and $\{A,n-r\}$ mark 
%the leading versus trailing
%singular spaces.     
The left singular spaces of a matrix $A$ are 
the right   singular spaces of its transpose $A^T$ and vice versa.
All matrix bases for the singular spaces $\mathbb S_{r,A}$ and $\mathbb T_{r,A}$
are given by the matrices $S_{r,A}X$ and
$T_{r,A}Y$, respectively,
for  nonsingular $r\times r$ matrices $X$ and $Y$.
The bases are 
orthogonal if the matrices $X$ and $Y$ 
are orthogonal.
%$B$ is an {\em approximate matrix basis} for 
%a linear space $\mathbb S$ within a relative error norm bound $\tau$
%if there exists a matrix $E$ such that $B+E$ is a matrix basis for 
%this space $\mathbb S$ and if $||E||\le \tau ||B||$.
%We use such bases in Sections \ref{sapsr}--\ref{svianv1}. 
%By dropping the last $m-\rho$ columns of the matrix $S_A$ and
%the last $n-\rho$ columns of the matrix $T_A$ for $\rho=\rank (A)$, we
% obtain {\em compact SVD} of the matrix $A$,
%\begin{equation}\label{eqcsvd}
%A=S_{\rho,A}\Sigma_{\rho,A}T_{\rho,A}^T.
%\end{equation}

%------------------------------------------------------------------------------

\subsection{The basic algorithm}\label{sprb0}

%------------------------------------------------------------------------------

Assume an $m\times n$ matrix $A$ 
having a small numerical rank $r$ and 
 a Gaussian $n\times r$
matrix $H$. Then according to 
\cite[Theorem 4.1]{HMT11}, 
the column span of the matrices $AH$ and  
$Q(AH)$ is likely to approximate the leading singular space $\mathbb S_{r,A}$
of the matrix $A$,  
and if it does, then it follows that the rank-$r$
matrix $QQ^TA$ approximates the matrix $A$.

In this subsection we recall the algorithm
supporting this theorem,  where temporarily we assume
nonrandom  multipliers $H$.
 In the next subsections we keep it nonrandom and
estimate the output approximation errors of the algorithm assuming
no oversampling,
suggested in \cite{HMT11}.  Then we
extend our study to the case where $H$ is a Gaussian, 
and  in Section \ref{sapgm} cover 
the results in the case of random structured 
multipliers.

%------------------------------------------------------------------------------

\begin{algorithm}\label{algbasmp} {\bf Low-rank approximation of a matrix.} 
(Cf. Remarks \ref{reaccr} and \ref{relss}.)
 %(cf. \cite[Section 10.3]{HMT11}).

%------------------------------------------------------------------------------

\begin{description}

%------------------------------------------------------------------------------

\item[{\sc Input:}] 
A
  matrix 
$ A\in \mathbb R^{m\times n}$,
its 
numerical rank $r$, and
two integers $p\ge 2$ and $l=r+p\le \min\{m,n\}$.
%(that is, the ratio $\sigma_{r+1}(A)/\sigma_{r}(A)$ is  small,
%but the ratio $||A||/\sigma_{r}(A)$) is not small),

%------------------------------------------------------------------------------

\item[{\sc Output:}]
an orthogonal matrix $Q\in \mathbb R^{m\times l}$ such that
%$\mathcal R(Q) \approx \mathbb T_{r,A}$ and 
the matrix
$QQ^TA\in \mathbb R^{m\times n}$ has rank 
at most $l$ and approximates the matrix $A$.

%------------------------------------------------------------------------------

\item[{\sc Initialization:}] $~$
Generate  an
$n\times l$
 matrix $H$.

%------------------------------------------------------------------------------ 

\item[{\sc Computations:}] $~$

%------------------------------------------------------------------------------

\begin{enumerate}

%------------------------------------------------------------------------------

\item %1
Compute an $n\times l$ orthogonal 
matrix $Q=Q(AH)$, sharing its range with 
the matrix $AH$.

\item %2
Compute
%||A-AQQ^T||\le \tau,
%\end{equation}
and output the matrix $R_{AH}A=QQ^TA$ and stop. 
%Otherwise output FAILURE and stop.

%------------------------------------------------------------------------------

\end{enumerate}

%------------------------------------------------------------------------------

\end{description}

%------------------------------------------------------------------------------

\end{algorithm}

This basic algorithm from \cite{HMT11} uses $O(lmn)$ flops overall.
 
%------------------------------------------------------------------------------

%\begin{remark}\label{revar34}
% One can  
%transpose the matrix $FA$ and 
%devise a
% dual variation of Algorithm \ref{algbasmp}, 
%which computes the orthogonal $(r+p)\times n$
% matrix  $Q=Q(FA)$ for a $(r+p)\times m$ pre-multiplier $F$
%and which
% approximates an orthogonal basis
%for the leading singular space $\mathbb T_{r,A}$. In this case
% the  matrix $(P_{FA}A^T)^T=AQ^TQ$
%of the rank $r+p$ would
%approximate the matrix $A$.
%\end{remark}

%------------------------------------------------------------------------------

\subsection{Analysis of the basic algorithm assuming 
no randomization and no oversampling}\label{saaba}

%------------------------------------------------------------------------------

In Corollaries \ref{coerrb} and  \ref{coover} of this subsection 
we estimate the error norms for the approximations
computed by Algorithm \ref{algbasmp} whose
 oversampling parameter $p$ is set to 0,
namely for the  
approximation of an orthogonal  basis for the leading singular space
$\mathbb S_{r,A}$ (by column set of the matrix $Q$ of the algorithm)
and for a
rank-$r$ approximation of the matrix $A$.
We first recall the following results.

\begin{theorem}\label{thsng} (Cf.  (\ref{eqopr}).)
Suppose  $A$ is
an $m\times n$ matrix,
$S_A\Sigma_AT_A^T$ is its SVD, 
%of (\ref{eqsvd}), 
 $r$ is an integer,
$0<r\le l\le\min\{m,n\}$, and 
$Q=Q_{r,A}$ is an orthogonal 
%orthogonal 
matrix 
basis for the space 
$\mathbb S_{r,A}$.
Then 
%\begin{equation}\label{eqlrap}
$||A-QQ^TA||=\sigma_{r+1}(A)$.
%\end{equation}
\end{theorem}

%------------------------------------------------------------------------------
%\subsection{A basis of 
%a leading singular space via randomized products}\label{sprb2}
%------------------------------------------------------------------------------ 
%------------------------------------------------------------------------------

\begin{theorem}\label{thover} 
Assume two
%three integers $m$, $n$, and $r$, 
%an $m\times n$ matrix
%we are given an upper bound $r_+$ on the numerical rank $r$ of
 %and the pair of  Gaussian 
matrices $A\in \mathbb R^{m\times n}$
and $H\in \mathbb R^{n\times r}$ and define the two matrices 
$A_r$ and $\bar A_r$ of (\ref{eqldtr}). Then 
 $AH=A_rH+\bar A_rH$ where $A_rH=S_{r,A}U$, $U=\Sigma_{r,A}T_{r,A}^TH$.
Furthermore
 the columns of the matrix $A_rH$ 
span the space $\mathbb S_{r,A}$
if $\rank (A_rH)=r$.
\end{theorem}

%------------------------------------------------------------------------------

These results together imply that 
the columns of the matrix $Q(AH)$ form an approximate 
orthogonal basis
of the linear space $\mathbb S_A$, and next we
estimate the error norms of this approximations.

\begin{theorem}\label{thuvwz} 
Keep the assumptions of Theorem \ref{thover}.  
Then

(i) $||\bar A_rH||_F\le \sigma_{r+1}(A)~||H||_F$.  

(ii) Furthermore if the matrix $T_{r,A}^TH$ is nonsingular, then
 $||(A_rH)^+||\le ||(T_{r,A}^TH)^{-1}||/\sigma_{r}(A)$. 
\end{theorem}

%------------------------------------------------------------------------------

\begin{proof}
Recall that
\begin{equation}\label{eqfrobun}
||U||=||U||_F=1,~||UAV||\le||A||,
~{\rm and}~||UAV||_F\le||A||_F~{\rm for~orthogonal~matrices}~U~{\rm and}~V. 
\end{equation} 
Then note that  
$||\bar A_rH||_F= ||\bar S_{A,r}\bar \Sigma_{A,r}\bar T_{A,r}^TH||_F\le
||\bar \Sigma_{A,r}\bar T_{A,r}^TH||_F$ by virtue of  bound (\ref{eqfrobun}).

Combine this bound with 
 Lemma \ref{lepr1}  and obtain that 
$||\bar A_rH||_F\le \sigma_{r+1}(A)~||\bar T_{A,r}^TH||_F$,
which is not greater than $\sigma_{r+1}(A)~||H||_F$
by virtue of  bound (\ref{eqfrobun}).
This proves part (i). 

Part (ii) follows because 
$(A_rH)^+=(S_{r,A}\Sigma_{r,A} T_{r,A}^TH)^{-1}=
(T_{r,A}^TH)^{-1}\Sigma_{r,A}^{-1}S_{r,A}^T$ if the matrix $T_{r,A}^TH$ is nonsingular
and because $||S_{r,A}||=1$ and $||\Sigma_{r,A}^{-1}||=1/\sigma_{r}(A)$. 
\end{proof}

%------------------------------------------------------------------------------

Combine  Theorems \ref{thpert1}, \ref{thover}, and \ref{thuvwz} to obtain the following estimates.

\begin{corollary}\label{coerrb} 
Keep the assumptions of Theorem \ref{thover}, let the matrix $T_{r,A}^TH$ be nonsingular
and write $$||E||_F=\sigma_{r+1}(A)~||H||_F,$$ 
$$\Delta_+=\sqrt 2~||E||_F~||(T_{r,A}^TH)^{-1}||~/\sigma_{r}(A)= 
\Delta_+=\sqrt 2~||H||_F~||(T_{r,A}^TH)^{-1}||~\sigma_{r+1}(A)/\sigma_{r}(A).$$
Then 
$$\Delta=||Q(A_rH)^T-Q(AH)^T|| \le \Delta_++O(||E||_F^2).$$ 
\end{corollary} 

 Next combine 
 Corollary \ref{copert} with Theorem \ref{thsng} and
employ the orthogonal projection $P_{AH}=Q(AH)Q(AH)^T$ (cf. (\ref{eqopr}))
to extend the latter estimate to bound the error norm of low-rank approximation of a matrix $A$ 
by means of Algorithm \ref{algbasmp}. 

%------------------------------------------------------------------------------

\begin{corollary}\label{coover} 
Keep the assumptions of Corollary \ref{coerrb} and write
%\begin{equation}\label{eqdlt+}
$\Delta_+'= \sigma_{r+1}(A)+2\Delta_+ ||A||$.
%\end{equation}
 Then 
%\begin{equation}\label{eqsgm} 
$$\Delta'=||A-P_{AH}A||\le \Delta_+'+O(||E||_F^2||A||).$$
%\end{equation}
\end{corollary}

%Let us prove this bound and specify it in terms of the parameters
%of Theorems \ref{thsignorm}, \ref{1} and  \ref{thover}
%and Corollary \ref{copert}.
\begin{proof}
Note that 
%$A-P_{AH}A=A-P_{M}A+(P_{M}-P_{AH})A$ and therefore 
$||A-P_{AH}A||\le ||A-P_{M}A||+||(P_{M}-P_{AH})A||$
for any $m\times r$ matrix $M$.
Write $M=A_rH$, apply Theorem \ref{thsng}
%(cf. Theorem \ref{thover})  
and  obtain
$||A-P_{M}A||=\sigma_{r+1}(A)$.
Corollaries \ref{copert}
and \ref{coerrb} together imply that
$||(P_{M}-P_{AH})A||\le ||A||~||P_{A_{r}H}-P_{AH}||\le 2 \Delta ||A||$.
 Combine the above relationships.
% to prove the corollary.
\end{proof}

% - - - - - - - - - - - - - - - - - - - - - - - - - - - - - - - - - - - - -

\begin{remark}\label{reaccr} 
Write $B_i=(A^TA)^iA$ and 
recall that $\sigma_j(B_i)=(\sigma_j(A))^{2i+1}$ for all positive integers
$i$ and $j$. Therefore
one can 
apply the power transforms
$A\rightarrow B_i$ 
for $i=1,2,\dots$ to
increase the
 ratio $\sigma_{r}(A)/\sigma_{r+1}(A)$, which
shows the gap between the two singular values.
Consequently
 the bound $\Delta_+$ on the error norm  of
the approximation of an orthogonal basis 
 of the leading singular space
$\mathbb S_{r,A}$ by $Q(B_iH)$ 
is expected to decrease as $i$
increases  (cf. \cite[equation (4.5)]{HMT11}).
We use the matrix $AH=B_0H$ in
Algorithm \ref{algbasmp}, but suppose we replace it with 
the matrices $B_iH$ for small positive integer $i$, or even for $i=1$,
which would amount just to symmetrization. Then 
we would obtain low-rank approximation with the optimum error
$\sigma_{r+1}(A)$ up to the terms of higher order in 
$\sigma_{r+1}(A)/\sigma_{r}(A)$
as long as the value $||H||_F||(T_{r,A}^TH)^{-1}||$
is reasonably bounded from above.  
The power transform
$A=B_0\rightarrow B_i$ requires to increase by a factor of $2i+1$
the number of ma\-trix-by-vec\-tor
multiplications involved, but for small positive integers $i$,
the additional  computational cost
is still dominated 
by  the costs of computing the SVD and rank-revealing factorizations.
% Theorem \ref{thuvwz}  
\end{remark}

% - - - - - - - - - - - - - - - - - - - - - - - - - - - - - - - - - - - - -

\begin{remark}\label{relss}
Let us summarize our analysis. Suppose
that
the ratio $\sigma_{r}(A)/\sigma_{r+1}(A)$
is large and that the matrix product
 $P=T_{r,A}^TH$ has full rank $r$
and is 
%reasonably 
well-con\-di\-tioned.
Now set to 0
 the oversampling integer parameter  $p$  
of Algorithm \ref{algbasmp}.
Then,
by virtue of 
Theorem \ref{thuvwz} and
Corollaries \ref{coerrb} and \ref{coover},
 the algorithm 
 outputs a close approximation $Q(AH)$
to an orthogonal bases for the leading singular space
$\mathbb S_{r,A}$
of the input matrix $A$ 
and  a rank-$r$ approximation
to this matrix. Up to the terms of higher order, the error norm 
of the latter approximation
is within a factor of $1+||H||_F~||(T_{r,A}^TH)^{-1}||/\sigma_{r}(A)$
from the optimal bound $\sigma_{r+1}(A)$. 
By applying the above power transform
of the input matrix $A$  
at a low computational  cost, 
 we can  decrease the error norm even below the value
 $\sigma_{r+1}(A)$.
\end{remark}

%------------------------------------------------------------------------------

\subsection{Supporting low-rank approximation with Gaussian multipliers}\label{sapg}

%------------------------------------------------------------------------------

In this subsection we extend the results of the previous one to support
the choice of Gaussian multiplier $H$ in Algorithm \ref{algbasmp}, whose 
``actual outcome 
%$\dots$ 
is very close to the typical outcome because of the measure 
concentration effect" \cite[page 226]{HMT11}.

\begin{theorem}\label{thrrk}
Suppose $A\in \mathbb R^{m\times n}$,
$A=S_A\Sigma_AT_A^T$ is its SVD of (\ref{eqsvd}),
$H=\mathbb R^{n\times r}$
is a Gaussian 
matrix, and $\rank(A)=\rho\ge r$.

(i) Then the matrix $T_{r,A}^TH$ is Gaussian.

(ii) Assume the values $\Delta_+$ and $\Delta_+'$ 
of Corollaries \ref{coerrb} and \ref{coover}
and the values $\nu_{F,n,r}$ and $\nu_{r,r}^+$
of Definition \ref{defnu}.
Then $\Delta_+=\sqrt 2~\nu_{F,n,r} \nu_{r,r}^+\sigma_{r+1}(A)/\sigma_{r}(A)~{\rm and}
 ~\Delta_+'= \sigma_{r+1}(A)+2\Delta_+ ||A||$.

\end{theorem}

%------------------------------------------------------------------------------

%------------------------------------------------------------------------------

\begin{proof}
 $T_{A}^TH$ is a Gaussian matrix by virtue of Lemma \ref{lepr3}.
Therefore so is its square submatrix $T_{r,A}^TH$ as well.  This proves part (i),
which implies part (ii). 
\end{proof}

%------------------------------------------------------------------------------%------------------------------------------------------------------------------

\begin{corollary}\label{coalg1}
A
Gaussian multiplier 
%(in particular Gaussian  multipliers)  
$H$ is expected 
to 
support safe numerical application of 
  Algorithm \ref{algbasmp} even
 where the oversampling integer parameter $p$ is set to 0.
\end{corollary}
\begin{proof}
Combine Theorems \ref{thdgr1} 
 and \ref{thrrk} with
Corollary \ref{cogfrwc}.
\end{proof}

%------------------------------------------------------------------------------

\subsection{Supporting low-rank approximation with random structured multipliers}\label{sapgm}

%------------------------------------------------------------------------------

 Multiplication of an $n\times n$ matrix $A$ by a
 Gaussian matrix $H$ at Stage 1
of Algorithm \ref{algbasmp} requires $(2r-1)n^2$ flops, 
but  
we can save a factor of $r/\log (r)$ flops by applying
 structured random multipliers $H$.
In particular we can use 
 subsampled random Fourier transforms (SRFTs) of
\cite[equation (4.6)]{HMT11},
subsampled random Hadamard transforms (SRHTs) of
 \cite{T11}, 
the chains of Givens rotations (CGRs) of
 \cite[Remark 4.5.1]{HMT11}, and the leading Toeplitz submatrices
 $C_{n,r}$ and $C_{r,n}$ of 
random circulant $n\times n$ matrices $C$.  
We need just $n$ random parameters to define 
a Gaussian circulant  $n\times n$ matrix $C$
and its leading Toeplitz blocks $C_{n,r}$ and $C_{r,n}$, 
and similarly for the other listed classes of structured matrices.

%------------------------------------------------------------------------------

\begin{example}\label{exsrft}
For two fixed integers $l$ and $n$, $1<l<n$, 
 SRFT  $n\times l$ matrices
are the matrices of the form $S=\sqrt{n/l}~D\Omega R$. Here 
$D$ is a random $n\times n$ diagonal
matrix whose diagonal entries are i.i.d. variables uniformly distributed on the 
unit circle $C(0,1)=\{x:~|x|=1\}$,  $\Omega$ is the DFT matrix,
and $R$ is a random $n\times l$
permutation matrix defined by random choice of $l$ columns under the uniform probability distribution
on the set of the $n$ columns
of the identity matrix $I_n$ (cf. \cite[equation (4.6) and Section 11]{HMT11}).
\end{example}

 Theorem \ref{thcpw} implies the following fact.

\begin{corollary}\label{cocpw}
Assume an $n\times l$ SRFT matrix $S$. Then 
 $\sqrt{l/n}~\Omega^{-1}S$ is an $n\times l$ submatrix of
a unitary circulant $n\times n$ matrix.
\end{corollary}
 
According to the extensive tests by many researchers,
various random structured $n\times l$ multipliers (such as SRFT, SRHT, CGR and CHR matrices)
support low-rank approximation already where 
the oversampling parameter $p=l-r$ is a reasonable constant
(see \cite{HMT11} and \cite{M11}).
In  particular 
 SRFT with oversampling by 20 is adequate in almost all
applications of low-rank approximations \cite[page 279]{HMT11}.
Likewise, in our extensive tests 
covered in Section \ref{stails}, 
Toeplitz multipliers defined as the $n\times r$ leading blocks of
$n\times n$ random circulant 
matrices
consistently  
 supported low-rank approximation
 without oversampling
as efficiently as Gaussian multipliers. 

As in the case of  our
randomized support for GENP and block Gaussian elimination,
formal analysis of
the impact of random structured multipliers 
is complicated because we cannot  use  Lemma \ref{lepr3}.
Nevertheless, by allowing substantial  oversampling,
one can still prove that SRFT multipliers 
are expected to support low-rank approximation 
of a matrix having  a small numerical rank.
     
\begin{theorem}\label{thsrtft}  {\rm Error bounds for
low-rank approximation with SRFT} (cf. \cite[Theorem 11.2]{HMT11}). 
Fix four integers $l$, $m$, $n$, and $r$ 
such that $4[\sqrt r+\sqrt {8\log(rn)n}]^2\log (r)\le l\le n$. 
Assume  an $m\times n$ matrix $A$ with singular
values $\sigma_1\ge\sigma_2\ge \sigma_3\ge \dots$, 
an $n\times l$ SRFT 
matrix $S$ of Example \ref{exsrft},
and $Y = AS$. 
Then with a probability  $1-O(1/r)$ it holds that
$$||(I-P_Y )A|| \le \sqrt{1+7n/l}~\sigma_{r+1}~{\rm and}
~||(I-P_Y )A||_F \le \sqrt{1+7n/l}~(\sum_{j>r} ~\sigma^2_j)^{1/2}.$$
\end{theorem}
 
\begin{remark}\label{remsrtft}
Clearly the theorem still holds
if we  replace the matrix $S$ by the matrix $US$ for 
a unitary matrix $U=(1/\sqrt n)\Omega^{-1}$.
In this case $US=CR$ for the matrix $R$ of Example 
\ref{exsrft} and the circulant matrix $C=\Omega^{-1}D\Omega$ 
 (cf. Theorem \ref{thcpw}).
By virtue of Theorem \ref{thsrtft} 
we can expect that Algorithm \ref{algbasmp} would produce a rank-$r$ approximation 
if we choose a multiplier $H$ being 
an SRFT $n\times l$ matrix  
or
the  $n\times l$ submatrix $CP$ of $n\times n$
 random unitary circulant matrix $C$ made up of its $l$ randomly selected
columns where the selection is defined by the matrix $P$
of Example \ref{exsrft} and where $l$ is an 
%sufficiently large 
integer of order $r\log(r)$. 
Recall that multiplication of an $n\times n$ Toeplitz matrix 
by an $n\times l$ matrix $US=CP$ involves $O(nl \log(n))$ flops \cite{p01},
versus $O(n^2l)$ in the straightforward algorithm.
\end{remark}

%------------------------------------------------------------------------------

\section{Numerical Experiments}\label{sexp}

% - - - - - - - - - - - - - - - - - - - - - - - - - - - - - - - - - - - - -
% XY changed for computers 
We performed numerical experiments with  random general,  circulant and Toeplitz  
matrices  by using MATLAB in the Graduate Center of the City University of New York 
on a Dell computer with a Intel Core 2 2.50 GHz processor and 4G memory running 
Windows 7. In particular we generated
Gaussian matrices by using the standard normal distribution function randn of MATLAB,
and we use the MATLAB function rand for generating numbers in the range $[0,1]$
under the uniform probability distribution function for Example \ref{exc1}. 
%The tests have been designed mostly by the first author and performed by coauthors.
We display our estimates obtained in terms of the spectral matrix norm but 
our tests showed similar results where we used the Frobenius norm  instead.

%------------------ - - - - - - - - - - - - - - - - - - - - - - - - - - - -

\subsection{GENP with Gaussian and random circulant multipliers}\label{sexgeneral}

%------------------ - - - - - - - - - - - - - - - - - - - - - - - - - - - -
 
We applied both GENP and the
preprocessed GENP to $n\times n$ DFT matrices $A=\Omega$
and to the matrices $A$ generate as follows.
We fixed $n=2^s$ and $k=n/2$ for $s=6,7,8,9,10$, and first,
by following \cite[Section 28.3]{H02},
generated a $k\times k$ matrix $A_k=U\Sigma V^T$ where we chose  $\Sigma=\diag (\sigma_i)_{i=1}^k$ with
$\sigma_i=1$ for $i=1,\dots,k-4$ and $\sigma_i=0$ for $i=k-3,\dots,k$ and 
where $U$ and $V$ were $k\times k$  random orthonormal  
matrices, computed as 
the $k\times k$ factors $Q(X)$ in the QR factorization of $k\times k$ random matrices $X$. 
Then we
 generated Gaussian  Toeplitz  matrices 
$B$, $C$ and $D$  such that 
 $||B||\approx ||C||\approx ||D||\approx ||A_k||\approx 1$
and defined 
the $n\times n$ matrix  
 $A=\begin{pmatrix}
A_k  &  B  \\
C    &  D
\end{pmatrix}.$
For every dimension $n$, $n=64, 128, 256, 512, 1024$  we run 1000 numerical tests
where we solved the linear system $A{\bf x}={\bf b}$ with  
Gaussian vector ${\bf b}$ and 
output the maximum, minimum and average relative residual norms 
$||A{\bf y}-{\bf b}||/||{\bf b}||$ as well as the standard deviation.
Figure \ref{fig:fig1} and Table \ref{tab61} show the norms of $A^{-1}$. 
They ranged from $2.2\times 10^1$ to  
$3.8\times 10^6$ in our tests. 

At first we describe the results of our tests for the latter class of matrices $A$.
As we expected GEPP has always output accurate solutions 
to the linear systems $A{\bf y}={\bf b}$
 in our tests (see Table \ref{tab62}). 
GENP, however, was expected to fail for these systems, 
because the $(n/2)\times (n/2)$   leading principal block $A_k$ of the matrix $A$  
was singular, having nullity 
$k-\rank (A_k)=4$.
Indeed  this caused poor performance of GENP in our tests,
which have consistently output corrupted solutions,   with  
relative 
residual norms
ranging from $10^{-3}$ to $10^2$. 

In view of Corollary \ref{cogmforalg} 
we expected to fix this deficiency by means of multiplication
by Gaussian matrices, and indeed  in all our tests
we observed 
residual norms below
 $1.3\times 10^{-6}$, and they decreased below $3.6\times 10^{-12}$ 
in a single step of iterative refinement
(see Table \ref{tab63}). Furthermore the tests 
showed the same power of preconditioning 
where we used the circulant multipliers of Examples \ref{exc1} and 
\ref{exunc} (see Tables  \ref{tab64} and \ref{tab65}). 
As can be expected, the output accuracy of GENP with 
preprocessing has deteriorated a little versus GEPP
 in our tests.
The output residual norms, however, were small enough to support
application of the inexpensive iterative refinement. Already its
single step decreased the average relative residual norm 
below $10^{-11}$ for $n=1024$ in all our tests with Gaussian multipliers
and to about $10^{-13}$ for $n=1024$ in all our tests with circulant multipliers
of Examples \ref{exc1} and \ref{exunc}.
%by average factors 
%ranging from $10^4$ to $10^5$ for Gaussian multipliers,
% from $10^3$ to $10^4$  for circulant multipliers and 
%from $10$ to $10^2$ for unitary circulant multipliers.
 See further details in Figures \ref{fig:fig2} and \ref{fig:fig3} and  Tables \ref{tab63}--\ref{tab65}. This indicates that GENP with preprocessing 
followed by even a single step of iterative refinement 
is backward stable, similarly to the celebrated result of
\cite{S80}.

We also applied similar tests to the
$n\times n$ DFT matrix $A=\Omega$. The results were in very good accordance with our study
in Section \ref{srnd}. Of course in this case the solution of a linear system $A{\bf x}={\bf b}$
can be computed immediately as ${\bf x}=\frac{1}{n}\Omega^H{\bf b}$,
but we were  not seeking the solution, but were trying to
 compare the performance of GENP with and without preprocessing.
In these tests the norm $||A^{-1}||$ was fixed at $1/\sqrt n$.
 GEPP produced the solution within the relative residual norm between 
$10^{-15}$ and $10^{-16}$, but GENP  
failed on the inputs $\Omega$
both when we used no preprocessing 
and used preprocessing with random 
circulant multipliers of Examples \ref{exc1} and \ref{exunc}. 
In these cases the
  relative residual norms of the output approximations ranged between
$10^{-2}$ and $10^{4}$. In contrast GENP applied to the
inputs preprocessed with Gaussian multipliers 
produced quite reasonable approximations to the solution.
Already after a single
step of iterative refinement,
they have at least matched the level of GEPP. 
Table \ref{tab65a} displays these norms
in some detail.

\begin{figure}
	\centering
		\includegraphics[width=0.80\textwidth]{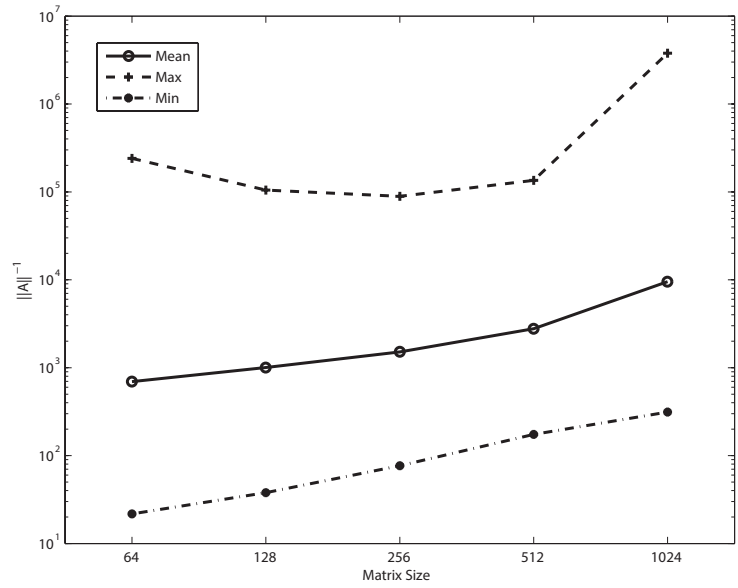}
	\caption{Norm of $A^{-1}$}	
	\label{fig:fig1}
\end{figure}

\begin{figure}
	\centering
		\includegraphics[width=0.80\textwidth]{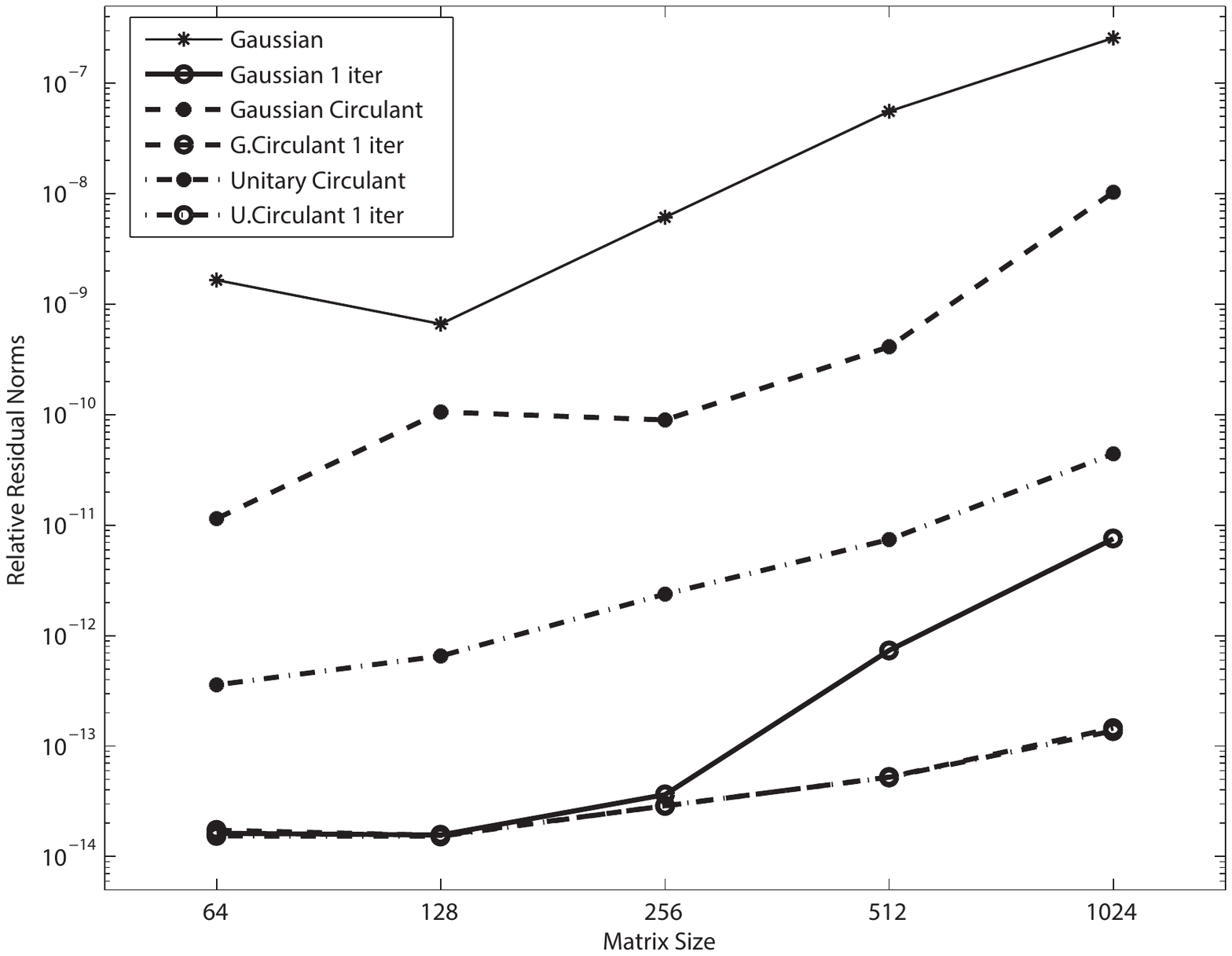}
	\caption{Average relative residual norms for GENP by using random multipliers.
	The two broken lines representing one iteration of circulant multipliers are overlapping at the bottom  of the display
	}
	\label{fig:fig2}
\end{figure}

\begin{figure}
	\centering
		\includegraphics[width=0.80\textwidth]{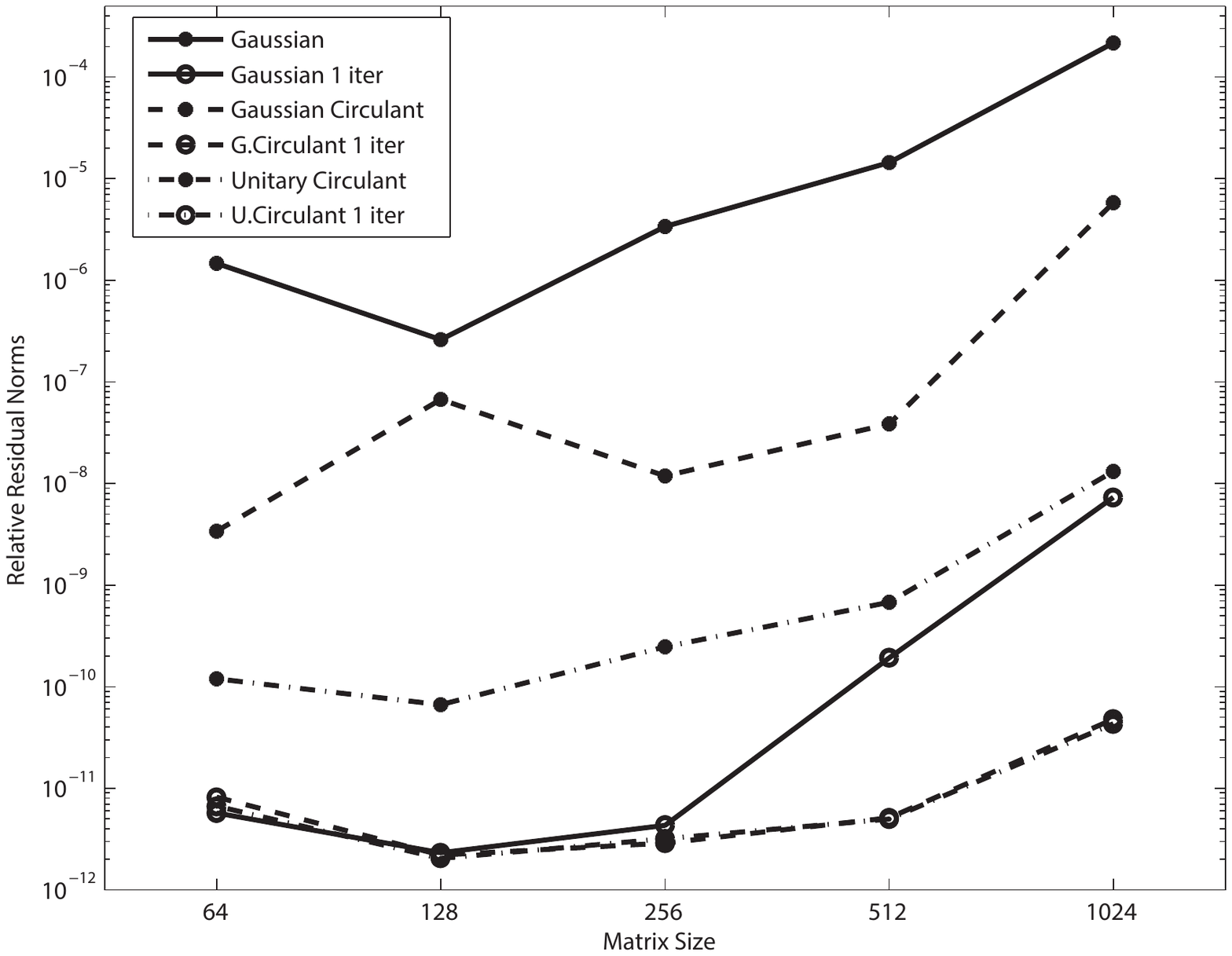}
	\caption{Maximum relative residual norms for GENP by using random multipliers.
	The two broken lines representing one iteration of circulant multipliers are overlapping at the bottom of the display
}
	\label{fig:fig3}
\end{figure}

%------------------------------------------------------------------------------

\subsection{Approximation of the leading singular spaces and 
low-rank appro\-xi\-ma\-tion of 
a matrix}\label{stails}  

%------------------------------------------------------------------------------

%We performed at most two refinement iterations for 
%the computed solution of every linear system of equations
%and matrix inverse. 

We approximated
the $r$-dimensional leading singular spaces of $n\times n$
matrices $A$ that have
numerical rank $r$, and we also
 approximated these matrices with matrices of rank $r$. 
For $n=64,128,256, 512,1024$ and  $r=8,32$ we  generated $n\times n$   random
orthogonal 
matrices $S$ and $T$ and diagonal matrices 
$\Sigma=\diag(\sigma_j)_{j=1}^n$ such that $\sigma_j=1/j,~j=1,\dots,r$,
$\sigma_j=10^{-10},~j=r+1,\dots,n$ (cf. \cite[Section 28.3]{H02}). 
Then we 
computed the input matrices
$A=S_A\Sigma_A T_A^T$, for which $||A||=1$ and $\kappa (A)=10^{10}$.
Furthermore 
we  generated 
 $n\times r$ random matrices $H$ 
and computed the matrices
$B_{r,A}=AH$, $Q_{r,A}=Q(B_{r,A})$, $S_{r,A}$, 
 $T_{r,A}$,  
$Y_{r,A}=Q_{r,A}^TS_{r,A}$, 
%(the latter five matrices by applying error-free  ring operations),
 and 
$Q_{r,A}Q_{r,A}^TA$.
Figures \ref{fig:fig4}--\ref{fig:fig7} and
Tables \ref{tab66}--\ref{tab614} display the resulting data on the 
residual norms 
${\rm rn}^{(1)}=||Q_{r,A}Y_{r,A}-S_{r,A}||$ and 
${\rm rn}^{(2)}=||A-Q_{r,A}Q_{r,A}^TA||$, 
%and ${\rm rn}^{(2)}_{F}=||A-Q_{r,A}Q_{r,A}^TA||_{F}$ 
obtained in 1000 runs of our tests
for every pair of $n$ and $r$. In these figures and tables 
${\rm rn}^{(1)}$ 
denotes the residual norms 
of the approximations of the matrix bases for the 
leading singular spaces $\mathbb S_{r,A}$, and
${\rm rn}^{(2)}$ denotes the residual norms 
of the approximations
of the matrix $A$ by the rank-$r$ matrix
$Q_{r,A}Q_{r,A}^TA$.

 Figures \ref{fig:fig4} and  \ref{fig:fig5} 
and Tables \ref{tab66}--\ref{tab613} show the norm ${\rm rn}^{(1)}$.  
The last column of each of the tables displays 
the ratio of the observed values ${\rm rn}^{(1)}$ 
 and its upper bound 
$\tilde \Delta_+= \sqrt 2~\frac{\sigma_{r+1}(A)}{\sigma_r(A)}||H||_F ||(T_{r,A}^TH)^{-1}||$ 
 estimated up to the higher order terms (cf. Corollary \ref{coerrb}). 
%{\bf (to X.Y.)}
%$\sigma_{r+1}(A)^2||F||_F^2$
In our tests we had  $\sigma_{r}(A)=1/r$ and $\sigma_{r+1}(A)=10^{-10}$.
Table \ref{tab66} covers the case where   
we generated  Gaussian multipliers $H$.
Tables \ref{tab610} and \ref{tab613} cover the cases where 
we generated 
random $n\times n$
circulant matrices of Examples \ref{exc1} and \ref{exunc}, respectively,
and applied their $n\times r$ Toeplitz leading blocks as multipliers $H$.

Figures \ref{fig:fig6} and \ref{fig:fig7}  and
Tables \ref{tab67}--\ref{tab614}
show similar results of our tests for the observed residual norms ${\rm rn}^{(2)}$
and their ratios with their upper bounds  $\tilde \Delta'_+=\sigma_{r+1}(A)+2\Delta_+ ||A||$, 
 estimated up to the higher order terms (cf. Corollary \ref{coover}). 

Tables \ref{tab69}--\ref{tab616} show some auxiliary information.
Namely, Table \ref{tab69} displays the data on the ratios $||(T_{r,A}^TH)^{-1}||/||(H_{r,r})^{-1}||$,
% of the norms of the inverses of the matrices $T_{r,A}^TH$ and $H_{r,r}$,
where $H_{r,r}$ denotes  the
$r\times r$ leading submatrix of the matrix $H$.
%{\bf (to X.Y.)}
%For GUOLIANG!
%The factor $\sqrt n$ in the second term of the upper bounds of
%in Corollary \ref{coover}, based on Theorem \ref{thpert1},
%and consequently in our upper bound $rn^{(2)}$ 
%reflects 
%the discrepancy between  the norms $||\cdot||$ and $||\cdot||_F$.
%So we tested whether this factor is indeed pertinent to our case and
%compared the computed values of $rn^{(2)}$ with the bound
%$\sigma_{q+1}(A)+2\sqrt {2}||\Delta'||~||(A^TF)^+||~||A||=\sigma_{q+1}(A)(1+2\sqrt {2}\nu_{r,n}||A||/\sigma_{q+1}(A))$. 
Tables \ref{tab616} and  \ref{tab617} display the average condition numbers 
of Gaussian $n\times n$ matrices and circulant
 $n\times n$  matrices $C$ of Example \ref{exc1},
respectively.
 
The test results are in quite good accordance with our theoretical study
of Gaussian multipliers and
suggest that the power of random circulant and Toeplitz multipliers is  similar 
to the power of Gaussian multipliers,
as in the case of various random structured multipliers
of \cite{HMT11} and \cite{M11}.

\begin{figure}
 \centering
	\includegraphics[width=0.80\textwidth]{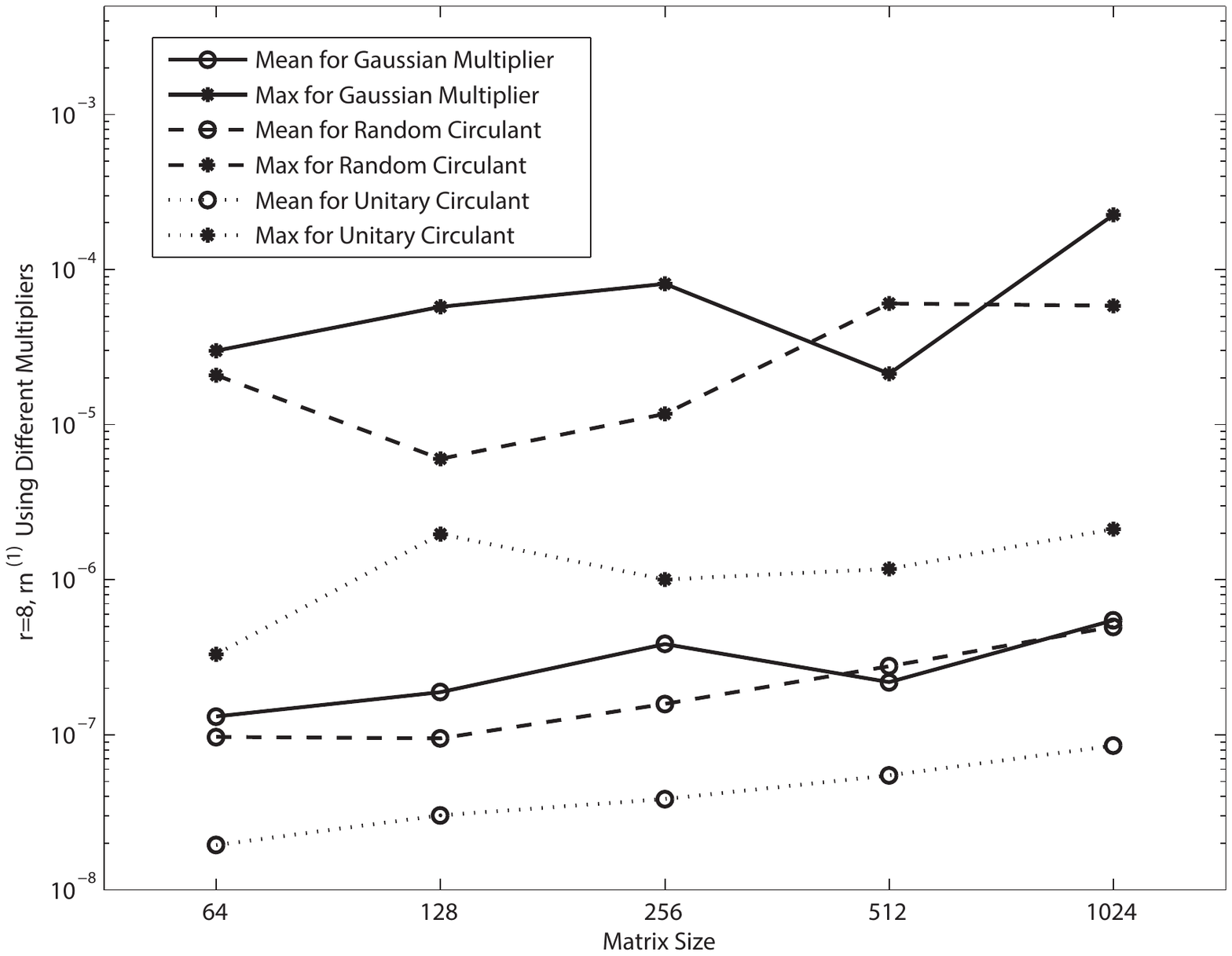}	
	\caption{Residual norms $rn^{(1)}$ using different random multipliers, case r=8}
	\label{fig:fig4}
\end{figure}

%For r=32, we plots the results of $rn^{(1)}$ for n=64,128,256,512,256 in Figure \ref{fig:tableRn1Q32}
 
\begin{figure}
	\centering
		\includegraphics[width=0.80\textwidth]{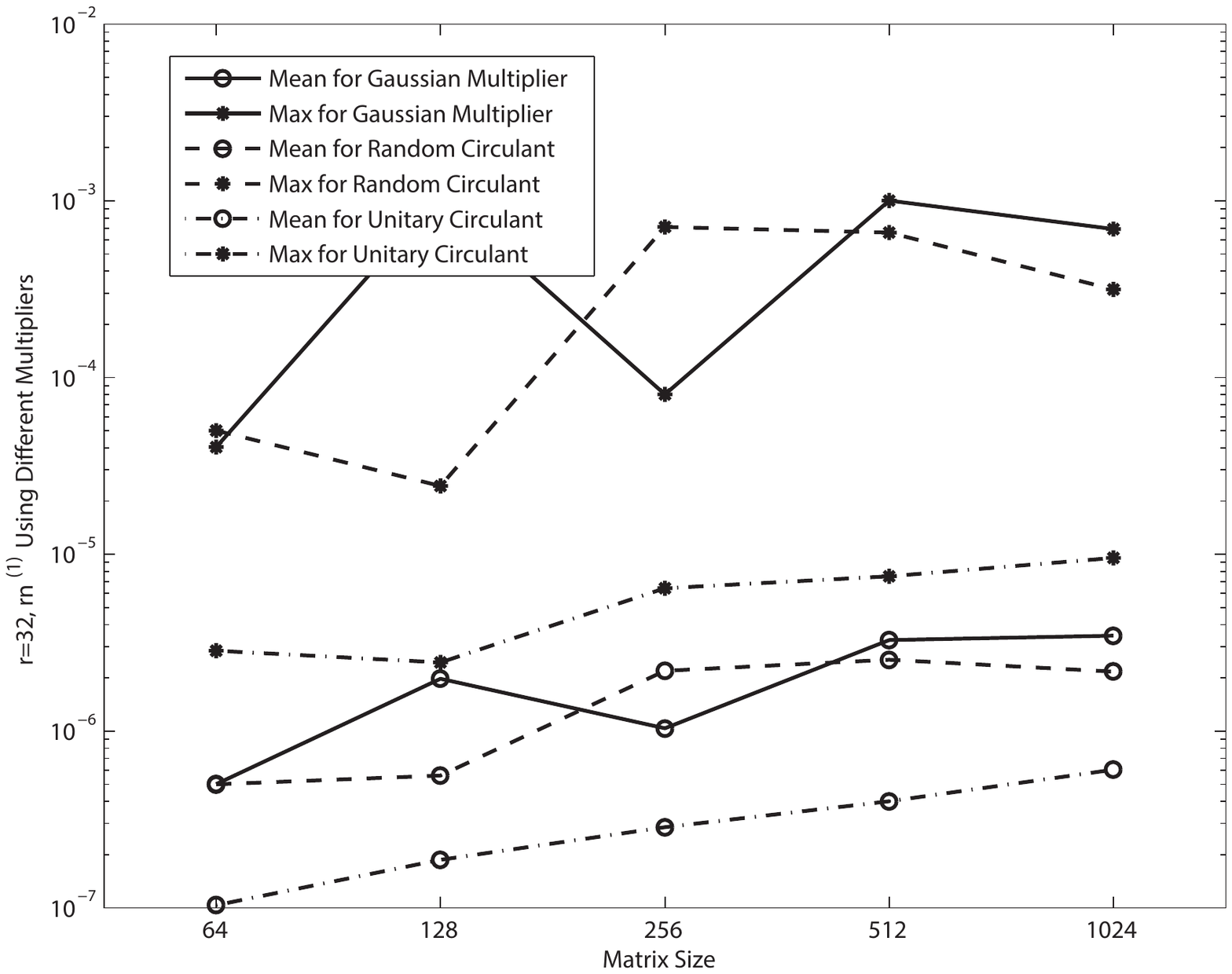}
	\caption{Residual norms $rn^{(1)}$ using different random multipliers, case r=32}
	\label{fig:fig5}
\end{figure}

%Similarly we also plot for q=8 and q=32 for $rn^{(2)}$. \ref{fig:tableRn2Q8}

\begin{figure}
	\centering
		\includegraphics[width=0.80\textwidth]{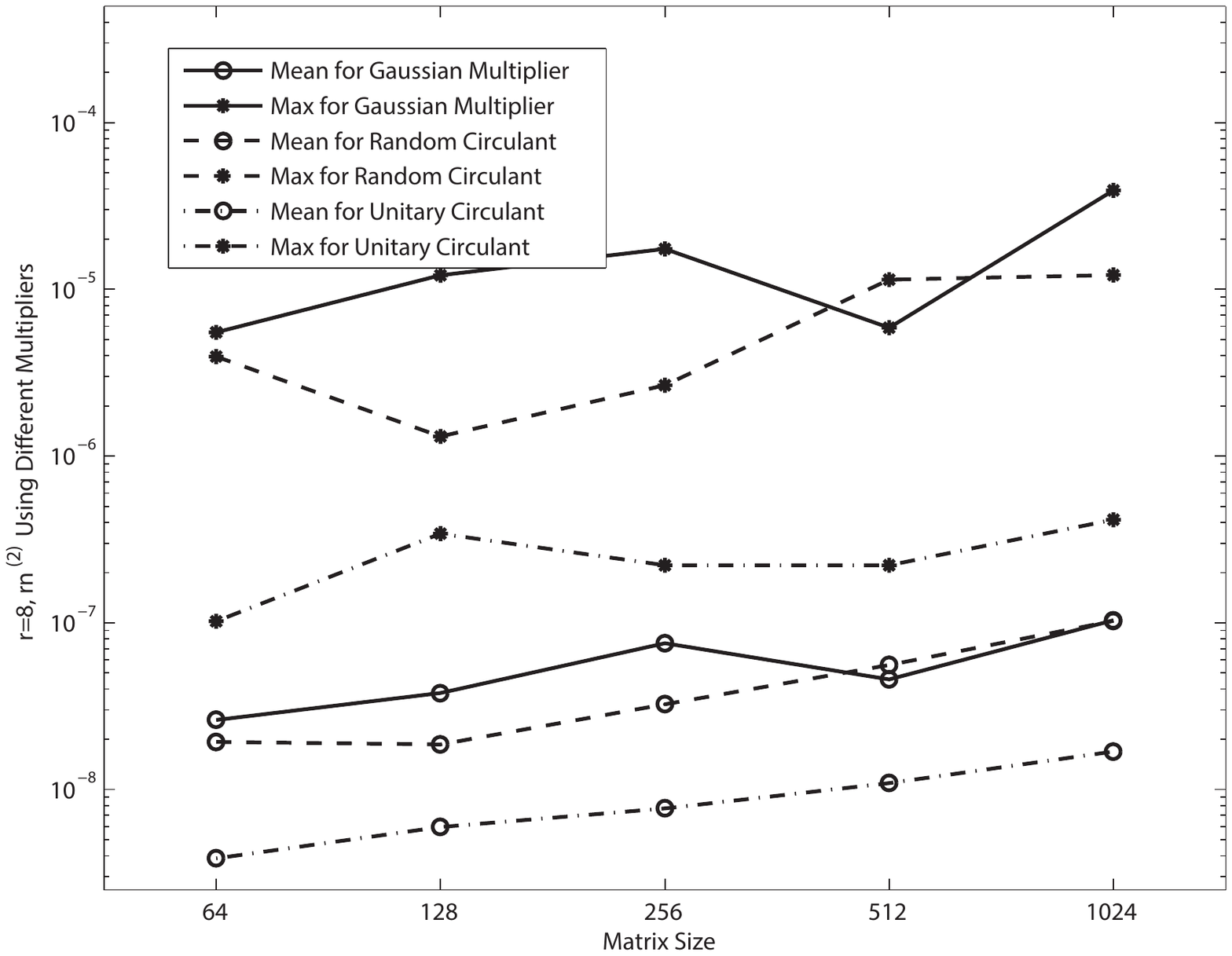}
	\caption{Residual norms  $rn^{(2)}$ using different random multipliers, case r=8}
	\label{fig:fig6}
\end{figure}

%For r=32, we plots the results of $rn^{(2)}$ for n=64,128,256,512,256 in the figure \ref{fig:tableRn2Q32}
 
\begin{figure}
	\centering
		\includegraphics[width=0.80\textwidth]{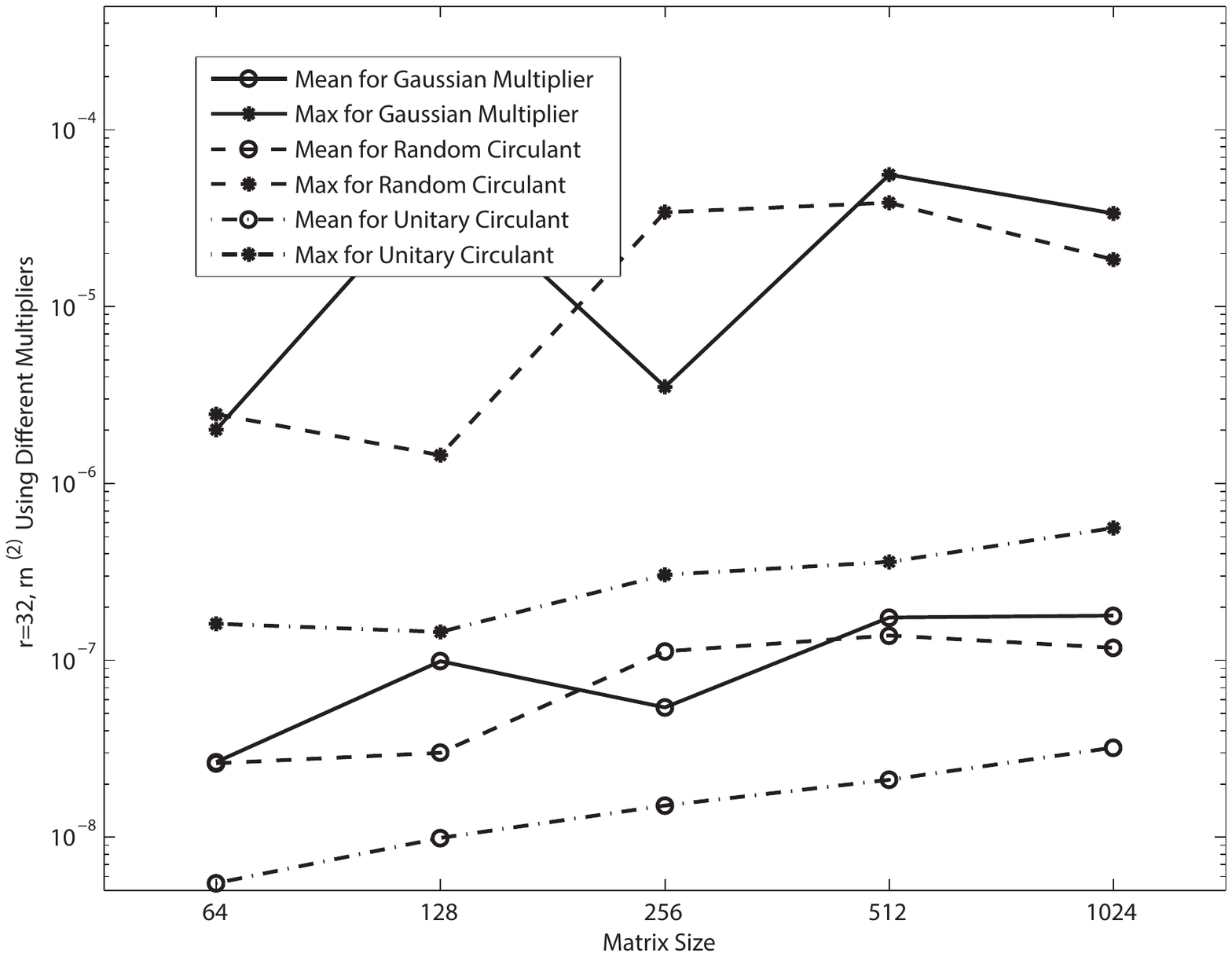}
	\caption{Residual norms $rn^{(2)}$ using different random multipliers, case r=32}
	\label{fig:fig7}
\end{figure}

%{(\bf to X.Y).}
% The plot ZZY also shows the results of these tests.
%The test results showed the ratios rn$^{(1)}/||\Delta'||$ 
%%and rn$^{(2)}/\delta'$ quite close to 1, even though
%we have not refined our approximations by using the power transform
%of \cite[equation (4.5)]{HMT11} and our Remark \ref{reaccr}.

\section{Conclusions}\label{sconcl}

% - - - - - - - - - - - - - - - - - - - - - - - - - - - - - - - - - - - - -

It is  known that a
 standard Gaussian random matrix
(we refer to it as Gaussian for short) 
has full rank with probability 1 and
is well-con\-di\-tioned with a probability
close to 1.
These properties  motivated our application of 
 Gaussian multipliers
to advancing 
 matrix 
computations.
In particular we preprocessed
 well-con\-di\-tioned nonsingular  input matrices  
by using Gaussian multipliers to support 
%the application of 
GENP (that is, Gaussian elimination with no pivoting) and block
 Gaussian elimination.  
These algorithms can readily fail in 
practical numerical computations without 
preprocessing, but we proved that we can avoid
these problems  
  with a
probability 
close to 1 if we preprocess the input matrix  by pre- or
 post-multiplying it by a
Gaussian matrix.

Our tests were in good accordance with that formal 
result, that is, we generated matrices that 
were hard for GENP, but the problems were consistently 
avoided when we preprocessed the inputs with Gaussian multipliers.
In that case a single loop of iterative refinement was always sufficient
to match the output accuracy of the customary algorithm of  GEPP,
 indicating that GENP with preprocessing 
followed by even a single step of iterative refinement 
is backward stable, similarly to the celebrated result of
\cite{S80}. 

In our tests we observed  similar results even where we applied  Gaussian 
 circulant (rather than Gaussian) multipliers.
Under
this choice we 
 generated only  $n$
random parameters for an $n\times n$
input, and the multiplication stage was
accelerated by a factor of $n/\log (n)$. 
The acceleration factor increases to $n^2/\log (n)$
 when  
the input matrix has the structure of Toeplitz type, but
we could support
numerical stabilization of GENP with Gaussian circulant multipliers  
only empirically.
Moreover, we proved that with a high probability 
Gaussian  circulant multipliers cannot fix numerical instability
 of the elimination algorithms 
for a  specific narrow class of inputs
 (see Theorem \ref{thcgenp}
and Remark  \ref{remext}).

This should motivate the search for 
%Would  supported with
alternative randomized structured multipliers
 that 
 would be expected to stabilize numerical performance 
of GENP and block Gaussian elimination
for any input or, say, for any 
Toeplitz and Toeplitz-like input matrix.
Among the candidate multipliers,
one can consider the products of random 
circulant and skew-circulant matrices,
possibly used as both pre- and post-multipliers.
Suppose that indeed they are expected to stabilize 
block Gaussian elimination numerically.
Then their support would be valuable for 
numerical application of 
 the MBA celebrated  
algorithm, because it is superfast 
for Toeplitz and Toeplitz-like input matrices
 and hence for their products with 
circulant and skew-circulant matrices 
(cf., e.g.,  \cite{B85},  \cite[Chapter 5]{p01},  and 
\cite{PQZ11}).  
 
We have extended our analysis to the problem of
rank-$r$ approximation of
%by examining the matrix $Q(AH)$
 an $m\times n$ matrix $A$ 
 having a numerical rank $r$.
With 
 a probability close to $1$
the column set of
the matrix  $AH$,
for  an $n\times (r+p)$ Gaussian matrix $H$
and a small positive
oversampling integer parameter $p$,
approximates a basis for the left  
leading singular
space $ \mathbb S_{r, A}$
associated with the $r$ largest singular values
of an $m\times n$ matrix $A$. 
Having  such an approximate basis 
available, one can readily approximate 
the matrix $A$ by a matrix of rank $r$.

This is an efficient, well developed algorithm
(see \cite{HMT11}), but we 
 proved that this algorithm is expected to
produce a reasonable rank-$r$ approximation 
with Gaussian multipliers already for $p=0$, that is, even 
without customary oversampling,
recommended in \cite{HMT11}.

Then again
%------------------------------------------------------------------------------
in our tests   
the latter techniques 
were  efficient
even where
instead of  Gaussian multipliers
we applied  random Toeplitz 
 multipliers, defined as the maximal leading submatrices of 
random circulant matrices. This has
accelerated the multiplication stage and has
limited
randomization to $n$
 parameters for an $n\times n$
input. 

Formal proof of the power of random structured SRFT 
multipliers with substantial oversampling
is known for low-rank approximation  \cite[Section 11]{HMT11},
and we  immediately extended it to the case when
the products of random unitary circulant 
multipliers and random rectangular permutation matrices 
were applied instead of the SRFT matrices (see Section \ref{srnd}).

A natural research challenge
 is the combination
of our randomized  
multiplicative  preprocessing with
randomized  augmentation and
additive preprocessing, studied 
 in  
\cite{PQ10}, \cite{PQ12},
\cite{PQZC},  \cite{PQZ13}, and \cite{PQZb}.
%   and \cite{PQZb}.

\bigskip

%------------------------------------------------------------------------------

{\bf Acknowledgements:}
Our research has been supported by NSF Grant CCF--1116736.
 and 
PSC CUNY Awards 64512--0042 and 65792--0043.
We are also grateful to the reviewer 
 for thoughtful helpful comments.

%$~$

%{\bf Acknowledgements} I wish to thank the reviewers for thoughtful and helpful comments.  

%$~$

%------------------------------------------------------------------------------

\bigskip
%\bigskip

%------------------------------------------------------------------------------

%\clearpage

%------------------------------------------------------------------------------

{\bf {\LARGE {Appendix}}}
\appendix 
%\bigskip

%------------------------------------------------------------------------------

%\section{Proof of Theorem \ref{thsng}}\label{sprf}

%------------------------------------------------------------------------------

%\begin{proof}
%Write $V=T_{q,A}T_{q,A}^T$ and $r=n-q$,
%so $T_A^TT_{q,A}=\begin{pmatrix}
%I_{q}  \\
%O_{r,q}
%\end{pmatrix}$,
%and obtain $AV=S_A\Sigma_AT_A^TT_{q,A}T_{q,A}^T=S_A\Sigma_A \begin{pmatrix}
%T_{q,A}^T  \\
%O_{r,q}
%\end{pmatrix}$, while 
%$A=S_A\Sigma_A \begin{pmatrix}
%T_{q,A}^T  \\
%T_{A,r}^T
%\end{pmatrix}$. Hence
%$A-AV=S_A\Sigma_A \begin{pmatrix}
%O_{q,n}  \\
%T_{A,r}^T
%\end{pmatrix}=S_A\diag(O_{q,q},\widehat\Sigma_q) T_A^T$,
%where $\widehat\Sigma_q$ is the $(m-q)\times (n-q)$ diagonal matrix
%with the diagonal entries $\sigma_{q+1},\dots,\sigma_l$.
%Thus
% $||A-AV||=||\widehat\Sigma_q||=\sigma_{q+1}(A)$
%because $S_A$ and $T_{A}$ are square orthogonal matrices.
%This proves  the claimed estimate
%for
%$B_T=T_{q,A}$, and we can immediately extend it 
% to any matrix basis $B_T=T_{q,A}U$ 
% for the space 
%$\mathbb T_{q,A}$ where $U$ is a 
%nonsingular matrix.  
%\end{proof}

%------------------------------------------------------------------------------

\section{Norm and condition of Gaussian matrices and
 the errors of randomized low-rank approximation}\label{sssm}

%--------------------------------------------------------------------------------

Let us reproduce some known bounds for the expected values
of the norms
and condition numbers of random matrices. 

\begin{theorem}\label{thszcd}
%  Let $G_{m,n}$ denote $m\times n$ Gaussian matrix. Then
(i) $\mathbb E (\nu_{n,n})\le 2\sqrt n$, 
%and Probability$(\nu_{n,n}\ge \alpha \sqrt n)\le c \exp(-c'\alpha^2n)$
(ii) $\mathbb E(\log(\kappa_{m,n}))\le \log (\frac{n}{m-n+1})+2.258$ for $m\ge n\ge 2$.
\end{theorem}
\begin{proof}
See \cite{S91} for part (i) and \cite[Theorem 6.1]{CD05} for part (ii).
\end{proof}
The bounds of part (i) of the theorem are quite tight (cf. Theorem \ref{thsignorm}).
The bounds of part (ii) imply the following more specific estimates. 
\begin{corollary}\label{cocd}
$\mathbb E(\log(\kappa_{n,n}))\le \log (n)+ 2.258$,
$\mathbb E(\kappa_{m,n})\le 5(1-1/k)$ for $k+1=\frac{m}{n-1}$ and $m\gg n\gg 1$.
\end{corollary}

The paper \cite{HMT11} proposed using Algorithm \ref{algbasmp}   
with the positive
oversampling integer parameter $p$ (see
 \cite[Algorithm 4.1 and Theorems 10.1 and 10.6]{HMT11}). This choice 
relied on the following bounds
of \cite[Theorems 10.5 and 10.6]{HMT11}
on the expected value $\mathbb E(||A-P_{AH}A||)$
of the output error norm of the  algorithm for $P_{AH}=QQ^T$,
\begin{equation}\label{eqerf}
\mathbb E(||A-P_{AH}A||_F)\le ((1+\frac{r}{p-1}\sum_{j>r}\sigma_{j}(A)^2)^{1/2},
\end{equation}
\begin{equation}\label{eqere}
\mathbb E(||A-P_{AH}A||)
\le ((1+\frac{r}{p-1}\sigma_{r+1}(A)^2)^{1/2}+
\frac{e\sqrt{r+p}}{p}\sum_{j>r}\sigma_{j}(A)^2)^{1/2}.
\end{equation}
Here is a simplified variant of the latter estimate from \cite[equation (1.8)]{HMT11},
\begin{equation}\label{eqeres}
\mathbb E(||A-P_{AH}A||)
\le (1+\frac{4\sqrt {r+p}}{p-1}\sqrt{\min\{m,n\}})\sigma_{r+1}(A).
\end{equation}

Quite typically the values $\sigma_j(A)$ for $j>r$
are not known, but one
can adapt the
 parameter $l$  
by using a posteriori error estimation. 
One can simplify this estimation by recalling   
from  \cite[equation (4.3)]{HMT11} that
\begin{equation}\label{eqrvm}
||A-P_{AH}A||\le 10\sqrt{2/\pi}\max_{j=1,\dots,r}(A-P_{AH}A)~{\bf g}_j
\end{equation}
with a probability at least $1-10^{-r}$.
Here ${\bf g}_j$ is the $j$th column of $n\times r$ Gaussian
matrix, 
that is, ${\bf g}_1,\dots,{\bf g}_r$ are $r$ independent  Gaussian vectors
of length $n$,
and $r$ is an integer parameter
(see our Remark \ref{reaccr} on improving this approximation).
Here is an alternative simplified expression from \cite[equation (1.9)]{HMT11},
\begin{equation}\label{eqrvms}
{\rm Probability}(||A-P_{AH}A||\le (1+9\sqrt{r+p}\sqrt{\min\{m,n\}})\sigma_{r+1}(A))\ge 1-3/p^p
\end{equation}
under some mild assumptions on the positive oversampling integer $p$.
The above bounds 
 show that 
low-rank  approximations of high quality can be obtained
by using a reasonably small oversampling integer parameter $p$, say $p=20$,
but they do not apply where $p\le 1$.
Our analysis of the basic algorithms 
relies on Corollary \ref{cor10} and provides
some reasonable formal support
 even where $p=0$. 

%------------------------------------------------------------------------------

\section{Uniform random sampling and nonsingularity of random matrices}\label{srsnrm}

%------------------------------------------------------------------------------

%Let $|\Delta|$ denote the cardinality of a  set $\Delta$ in  any fixed ring. 
{\em Uniform random sampling} of elements from a finite set $\Delta$ is their selection   
from  this set at random, independently of each other and
under the uniform probability distribution on the set $\Delta$. 

%------------------------------------------------------------------------------

%Let $|\Delta|$ denote the cardinality of a  set $\Delta$ in  any fixed ring. 
The total degree of a multivariate monomial is the sum of its degrees
in all its variables. The total degree of a polynomial is the maximal total degree of 
its monomials.

%------------------------------------------------------------------------------

\begin{lemma}\label{ledl} \cite{DL78}, \cite{S80a}, \cite{Z79}.
For a set $\Delta$ of a cardinality $|\Delta|$ in any fixed ring  
let a polynomial in $m$ variables have a total degree $d$ and let it not vanish 
identically on the set $\Delta^m$. Then the polynomial vanishes in at most 
$d|\Delta|^{m-1}$ points of this set. 
\end{lemma}

\begin{theorem}\label{thdl} 
Under the assumptions of Lemma \ref{ledl} let the values of the variables 
of the polynomial be randomly and uniformly sampled from a finite set $\Delta$. 
Then the polynomial vanishes with a probability at most $\frac{d}{|\Delta|}$. 
\end{theorem}

%------------------------------------------------------------------------------

\begin{corollary}\label{codlstr} 
Let the entries of a general or Toeplitz  $m\times n$ 
matrix have been randomly and uniformly 
sampled from a finite set $\Delta$ of cardinality $|\Delta|$ (in any fixed ring). 
Let $l=\min\{m,n\}$.
Then (a) every $k\times k$ submatrix $M$ for $k\le l$ is nonsingular with a probability at 
least $1-\frac{k}{|\Delta|}$ and (b) is strongly nonsingular with a probability at least 
$1-\sum_{i=1}^k\frac{i}{|\Delta|}= 1-\frac{(k+1)k}{2|\Delta|}$.
%Furthermore (c) if the submatrix $M$ is indeed nonsingular, then any entry of its inverse is 
%nonzero with a probability at least $1-\frac{k-1}{|\Delta|}$.
\end{corollary}

%------------------------------------------------------------------------------

\begin{proof}
Clearly the claims of the corollary  hold for generic matrices. 
Now note that the singularity of a $k\times k$ matrix means that its determinant vanishes,
but the determinant is a polynomial of total degree $k$ in the entries. Therefore
Theorem \ref{thdl} implies
parts (a) and consequently (b). Part (c) follows because a fixed entry of the inverse vanishes
if and only if the respective entry of the adjoint vanishes, but up to the sign the latter 
entry is the determinant of a $(k-1)\times (k-1)$ submatrix of the input matrix $M$, and so it is
a polynomial of degree $k-1$ in its entries. 
\end{proof}

%------------------------------------------------------------------------------

\section{Perturbation errors of matrix inversion}\label{sosvdi0}

%------------------------------------------------------------------------------

\begin{theorem}\label{thpert}  \cite[Corollary 1.4.19]{S98}.
Assume a pair of square matrices $A$ (nonsingular) and $E$ such that 
$||A^{-1}E||< 1$. Then $||(A+E)^{-1}||\le\frac{||A^{-1}||}{1-||A^{-1}E||}$ and 
 $\frac{||(A+E)^{-1}-A^{-1}||}{||A^{-1}||}\le \frac{||A^{-1}||}{1-||A^{-1}E||}$.
\end{theorem}

\begin{theorem}\label{thpert1}  \cite[Theorem 5.1]{S95}.
%For two matrices $A$ and $E$ in $\mathbb R^{
%Let  $Q(A)$ and $Q(A+E)$ denote the Q-factors
Assume a pair of
 $m\times n$ matrices $A$ and $A+E$, 
%respectively,
and let the norm $||E||$ be small. Then 
$||Q(A+E)-Q(A)||\le \sqrt 2 ||A^+||~||E||_F+O(||E||_F^2)$.
\end{theorem}

%------------------------------------------------------------------------------

$P_A$  denotes the {\em orthogonal projector} on the range
of a matrix $A$ having full column rank,

\begin{equation}\label{eqopr}
P_A=A(A^TA)^{-1}A^T=AA^+=QQ^T~{\rm for}~Q=Q(A).
\end{equation}
%Theorem \ref{thpert1}  implies the following
% perturbation norm bound.

%------------------------------------------------------------------------------

\begin{corollary}\label{copert}  
Suppose $m\times n$ matrices $A$ and $A+E$ have full rank.
Then  

$||P_{A+E}-P_A||\le 2 ||Q(A+E)-Q(A)||\le 2 \sqrt 2~||A^+||~||E||_F+O(||E||_F^2).$
\end{corollary}

\begin{proof}
Clearly 
%\begin{eqnarray*}
$P_{A+E}-P_A=Q(A+E)Q(A+E)^T-Q(A)Q(A)^T=$
%\end{eqnarray*}
\begin{eqnarray*}
(Q(A+E)-Q(A))Q(A+E)^T+Q(A)(Q(A+E)^T-Q(A)^T).
\end{eqnarray*}
Consequently

$||P_{A+E}-P_A||\le 
||Q(A+E)-Q(A)||~||Q(A+E)^T||+||Q(A)||~||Q(A+E)^T-Q(A)^T||$.

Substitute $||Q(A)||=||Q(A+E)^T||=1$ and $||Q(A+E)^T-Q(A)^T||=||Q(A+E)-Q(A)||$
and obtain that $||P_{A+E}-P_A||\le 2 ||Q(A+E)-Q(A)||$.
Substitute the bound of Theorem \ref{thpert1}.
%and obtain the claimed bound on the norm  $||P_{A+E}-P_A||$.
\end{proof}

%------------------------------------------------------------------------------

\section{Tables}\label{stab}

%------------------------------------------------------------------------------
%------------------------------------------------------------------------------

\begin{table}[ht]
  \caption{The norms $||A||^{-1}$ of the input matrices $A$}
  \label{tab61}
  \begin{center}
    \begin{tabular}{| c | c | c | c | c | c |}
      \hline
      \bf{dimension} &  \bf{mean} & \bf{max} & \bf{min} & \bf{std} \\ \hline
%w/o pivoting w/ preconditioning
 $64$ &  $6.95\times 10^2$ & $2.41\times 10^5$ & $2.18\times 10^1$ & $7.87\times 10^3$ \\ \hline
 $128$ &  $1.00\times 10^3$ & $1.05\times 10^5$ & $3.78\times 10^1$ & $5.81\times 10^3$ \\ \hline
 $256$ &  $1.51\times 10^3$ & $8.90\times 10^4$ & $7.68\times 10^1$ & $6.06\times 10^3$ \\ \hline
 $512$ &  $2.78\times 10^3$ & $1.35\times 10^5$ & $1.74\times 10^2$ & $8.64\times 10^3$ \\ \hline
 $1024$ & $9.54\times 10^3$ & $3.79\times 10^6$ & $3.13\times 10^2$ & $1.21\times 10^5$ \\ \hline
% new 1000 results April 2014
%  6.9565e+002  2.4056e+005  2.1765e+001  7.8710e+003
%  1.0046e+003  1.0479e+005  3.7846e+001  5.8103e+003
%  1.5126e+003  8.9030e+004  7.6814e+001  6.0575e+003
%  2.7761e+003  1.3521e+005  1.7428e+002  8.6402e+003
%  9.5367e+003  3.7860e+006  3.1316e+002  1.2117e+005

    \end{tabular}
  \end{center}
\end{table}
   
%table ends..................................................................................

%------------------------------------------------------------------------------

\begin{table}[ht]
  \caption{Relative residual norms  of GEPP}
  \label{tab62}
  \begin{center}
    \begin{tabular}{| c | c | c | c | c | c |}
      \hline
      \bf{dimension}  & \bf{mean} & \bf{max} & \bf{min} & \bf{std} \\ \hline
%w pivoting w/o preconditioning
 $64$ &  $4.91\times 10^{-14}$ & $2.06\times 10^{-11}$ & $1.75\times 10^{-15}$ & $6.64\times 10^{-13}$ \\ \hline
 $128$ &  $6.86\times 10^{-14}$ & $7.58\times 10^{-12}$ & $3.97\times 10^{-15}$ & $3.02\times 10^{-13}$ \\ \hline
 $256$  & $2.00\times 10^{-13}$ & $1.95\times 10^{-11}$ & $1.05\times 10^{-14}$ & $8.93\times 10^{-13}$ \\ \hline
 $512$ &  $6.08\times 10^{-13}$ & $5.76\times 10^{-11}$ & $3.55\times 10^{-14}$ & $2.65\times 10^{-12}$ \\ \hline
 $1024$  & $2.67\times 10^{-12}$ & $8.02\times 10^{-10}$ & $1.13\times 10^{-13}$ & $2.65\times 10^{-11}$ \\ \hline
  \end{tabular}
  \end{center}
\end{table}

% new 1000 April 2nd
%  4.9145e-014  2.0642e-011  1.7458e-015  6.6398e-013
%  6.8628e-014  7.5834e-012  3.9700e-015  3.0232e-013
%  2.0001e-013  1.9515e-011  1.0534e-014  8.9319e-013
%  6.0808e-013  5.7555e-011  3.5467e-014  2.6473e-012
%  2.6729e-012  8.0221e-010  1.1303e-013  2.6546e-011
%table ends..................................................................................

%------------------------------------------------------------------------------

\begin{table}[ht]
  \caption{Relative residual norms:
 GENP with Gaussian multipliers}
  \label{tab63}
  \begin{center}
    \begin{tabular}{| c | c | c | c | c | c | c |}
      \hline
      \bf{dimension} & \bf{iterations} & \bf{mean} & \bf{max} & \bf{min} & \bf{std} \\ \hline
%w/o pivoting w/ preconditioning
 $64$ & $0$ & $1.66\times 10^{-9}$ & $1.47\times 10^{-6}$ & $4.47\times 10^{-14}$ & $4.67\times 10^{-8}$ \\ \hline
 $64$ & $1$ & $1.63\times 10^{-14}$ & $5.71\times 10^{-12}$ & $5.57\times 10^{-16}$ & $1.91\times 10^{-13}$ \\ \hline
 $128$ & $0$ & $6.62\times 10^{-10}$ & $2.61\times 10^{-7}$ & $3.98\times 10^{-13}$ & $8.66\times 10^{-9}$ \\ \hline
 $128$ & $1$ & $1.57\times 10^{-14}$ & $2.31\times 10^{-12}$ & $9.49\times 10^{-16}$ & $8.23\times 10^{-14}$ \\ \hline
 $256$ & $0$ & $6.13\times 10^{-9}$ & $3.39\times 10^{-6}$ & $2.47\times 10^{-12}$ & $1.15\times 10^{-7}$ \\ \hline
 $256$ & $1$ & $3.64\times 10^{-14}$ & $4.32\times 10^{-12}$ & $1.91\times 10^{-15}$ & $2.17\times 10^{-13}$ \\ \hline
 $512$ & $0$ & $5.57\times 10^{-8}$ & $1.44\times 10^{-5}$ & $1.29\times 10^{-11}$ & $7.59\times 10^{-7}$ \\ \hline
 $512$ & $1$ & $7.36\times 10^{-13}$ & $1.92\times 10^{-10}$ & $3.32\times 10^{-15}$ & $1.07\times 10^{-11}$ \\ \hline
 $1024$ & $0$ & $2.58\times 10^{-7}$ & $2.17\times 10^{-4}$ & $4.66\times 10^{-11}$ & $6.86\times 10^{-6}$ \\ \hline
 $1024$ & $1$ & $7.53\times 10^{-12}$ & $7.31\times 10^{-9}$ & $6.75\times 10^{-15}$ & $2.31\times 10^{-10}$ \\ \hline
 %only added 64,256,1024 from below
 
    \end{tabular}
  \end{center}
\end{table}

%table ends..................................................................................
% Data April 2014
% Table63A
%  1.6567e-009  1.4704e-006  4.4691e-014  4.6680e-008
%  6.6220e-010  2.6087e-007  3.9742e-013  8.6600e-009
%  6.1273e-009  3.3854e-006  2.4667e-012  1.1479e-007
%  5.5676e-008  1.4427e-005  1.2907e-011  7.5934e-007
%  2.5759e-007  2.1678e-004  4.6553e-011  6.8583e-006
% with iter
%  1.6267e-014  5.7133e-012  5.5739e-016  1.9083e-013
%  1.5666e-014  2.3139e-012  9.4879e-016  8.2277e-014
%  3.6431e-014  4.3223e-012  1.9135e-015  2.1671e-013
%  7.3560e-013  1.9239e-010  3.3189e-015  1.0689e-011
%  7.5289e-012  7.3094e-009  6.7485e-015  2.3114e-010

%------------------------------------------------------------------------------

\begin{table}[ht]
  \caption{Relative residual norms:
 GENP with real circulant
 Gaussian multipliers of Example \ref{exc1}}
  \label{tab64}
  \begin{center}
    \begin{tabular}{| c | c | c | c | c | c | c |}
      \hline
      \bf{dimension} & \bf{iterations} & \bf{mean} & \bf{max} & \bf{min} & \bf{std} \\ \hline
%w/o pivoting w/ preconditioning
 $64$ & $0$ & $1.15\times 10^{-11}$ & $3.39\times 10^{-9}$ & $2.15\times 10^{-14}$ & $1.18\times 10^{-10}$ \\ \hline
 $64$ & $1$ & $1.73\times 10^{-14}$ & $8.18\times 10^{-12}$ & $5.95\times 10^{-16}$ & $2.62\times 10^{-13}$ \\ \hline
 $128$ & $0$ & $1.06\times 10^{-10}$ & $6.71\times 10^{-8}$ & $1.73\times 10^{-13}$ & $2.15\times 10^{-9}$ \\ \hline
 $128$ & $1$ & $1.56\times 10^{-14}$ & $2.20\times 10^{-12}$ & $8.96\times 10^{-16}$ & $7.91\times 10^{-14}$ \\ \hline
 $256$ & $0$ & $8.97\times 10^{-11}$ & $1.19\times 10^{-8}$ & $6.23\times 10^{-13}$ & $4.85\times 10^{-10}$ \\ \hline
 $256$ & $1$ & $2.88\times 10^{-14}$ & $2.89\times 10^{-12}$ & $1.89\times 10^{-15}$ & $1.32\times 10^{-13}$ \\ \hline
 $512$ & $0$ & $4.12\times 10^{-10}$ & $3.85\times 10^{-8}$ & $2.37\times 10^{-12}$ & $2.27\times 10^{-9}$ \\ \hline
 $512$ & $1$ & $5.24\times 10^{-14}$ & $5.12\times 10^{-12}$ & $2.95\times 10^{-15}$ & $2.32\times 10^{-13}$ \\ \hline
 $1024$ & $0$ & $1.03\times 10^{-8}$ & $5.80\times 10^{-6}$ & $1.09\times 10^{-11}$ & $1.93\times 10^{-7}$ \\ \hline
 $1024$ & $1$ & $1.46\times 10^{-13}$ & $4.80\times 10^{-11}$ & $6.94\times 10^{-15}$ & $1.60\times 10^{-12}$ \\ \hline    
 \end{tabular}
  \end{center}
\end{table}
%table ends..................................................................................
% XY April 2014
% table 63C
%  1.1453e-011  3.3901e-009  2.1486e-014  1.1801e-010
%  1.0563e-010  6.7134e-008  1.7322e-013  2.1503e-009
%  8.9678e-011  1.1878e-008  6.2334e-013  4.8455e-010
%  4.1239e-010  3.8524e-008  2.3715e-012  2.2744e-009
%  1.0340e-008  5.8000e-006  1.0943e-011  1.9289e-007

% table63D

%  1.7327e-014  8.1785e-012  5.9459e-016  2.6158e-013
%  1.5558e-014  2.2015e-012  8.9603e-016  7.9100e-014
%  2.8755e-014  2.8901e-012  1.8867e-015  1.3188e-013
%  5.2426e-014  5.1157e-012  2.9463e-015  2.3194e-013
%  1.4548e-013  4.7950e-011  6.9379e-015  1.5792e-012
%------------------------------------------------------------------------------

\begin{table}[ht]
  \caption{Relative residual norms:
 GENP with unitary circulant
multipliers of Example \ref{exunc}}
  \label{tab65}
  \begin{center}
    \begin{tabular}{| c | c | c | c | c | c | c |}
      \hline
      \bf{dimension} & \bf{iterations} & \bf{mean} & \bf{max} & \bf{min} & \bf{std} \\ \hline
%w/o pivoting w/ preconditioning
 $64$ & $0$ & $3.59\times 10^{-13}$ & $1.19\times 10^{-10}$ & $6.14\times 10^{-15}$ & $3.95\times 10^{-12}$ \\ \hline
 $64$ & $1$ & $1.53\times 10^{-14}$ & $6.69\times 10^{-12}$ & $5.74\times 10^{-16}$ & $2.14\times 10^{-13}$ \\ \hline
 $128$ & $0$ & $6.54\times 10^{-13}$ & $6.64\times 10^{-11}$ & $2.68\times 10^{-14}$ & $2.67\times 10^{-12}$ \\ \hline
 $128$ & $1$ & $1.53\times 10^{-14}$ & $2.04\times 10^{-12}$ & $9.31\times 10^{-16}$ & $7.45\times 10^{-14}$ \\ \hline
 $256$ & $0$ & $2.37\times 10^{-12}$ & $2.47\times 10^{-10}$ & $9.41\times 10^{-14}$ & $1.06\times 10^{-11}$ \\ \hline
 $256$ & $1$ & $2.88\times 10^{-14}$ & $3.18\times 10^{-12}$ & $1.83\times 10^{-15}$ & $1.36\times 10^{-13}$ \\ \hline
 $512$ & $0$ & $7.42\times 10^{-12}$ & $6.77\times 10^{-10}$ & $3.35\times 10^{-13}$ & $3.04\times 10^{-11}$ \\ \hline
 $512$ & $1$ & $5.22\times 10^{-14}$ & $4.97\times 10^{-12}$ & $3.19\times 10^{-15}$ & $2.29\times 10^{-13}$ \\ \hline
 $1024$ & $0$ & $4.43\times 10^{-11}$ & $1.31\times 10^{-8}$ & $1.28\times 10^{-12}$ & $4.36\times 10^{-10}$ \\ \hline
 $1024$ & $1$ & $1.37\times 10^{-13}$ & $4.33\times 10^{-11}$ & $6.67\times 10^{-15}$ & $1.41\times 10^{-12}$ \\ \hline
   \end{tabular}
  \end{center}
\end{table}
%table ends..................................................................................
% XY April 2014 1000 run
% table63UC
%  3.5929e-013  1.1932e-010  6.1355e-015  3.9506e-012
%  6.5449e-013  6.6357e-011  2.6822e-014  2.6707e-012
%  2.3728e-012  2.4716e-010  9.4129e-014  1.0626e-011
%  7.4207e-012  6.7706e-010  3.3516e-013  3.0429e-011
%  4.4275e-011  1.3149e-008  1.2799e-012  4.3604e-010
% table63UD
%  1.5337e-014  6.6868e-012  5.7420e-016  2.1369e-013
%  1.5293e-014  2.0427e-012  9.3103e-016  7.4470e-014
%  2.8844e-014  3.1800e-012  1.8347e-015  1.3565e-013
%  5.2246e-014  4.9692e-012  3.1905e-015  2.2918e-013
%  1.3745e-013  4.3257e-011  6.6728e-015  1.4112e-012

\begin{table}[ht]
  \caption{Relative residual norms:
 GENP with Gaussian multipliers and iterative refinement}
  \label{tab65a}
  \begin{center}
  \begin{tabular}{| c |c|  c | c |  c |c|}
      \hline
  \bf{dimension} & \bf{iterations} & \bf{mean} & \bf{max} & \bf{min} & \bf{std} \\ \hline
  
 $64$ & $0$ & $3.41\times 10^{-13}$ & $1.84\times 10^{-11}$ & $1.73\times 10^{-14}$ & $1.84\times 10^{-12}$ \\ \hline
 $64$ & $1$ & $5.10\times 10^{-16}$ & $8.30\times 10^{-16}$ & $4.02\times 10^{-16}$ & $6.86\times 10^{-17}$ \\ \hline
 $128$ & $0$ & $5.48\times 10^{-13}$ & $7.21\times 10^{-12}$ & $6.02\times 10^{-14}$ & $9.05\times 10^{-13}$ \\ \hline
 $128$ & $1$ & $7.41\times 10^{-16}$ & $9.62\times 10^{-16}$ & $6.11\times 10^{-16}$ & $6.82\times 10^{-17}$ \\ \hline
 $256$ & $0$ & $2.26\times 10^{-12}$ & $4.23\times 10^{-11}$ & $2.83\times 10^{-13}$ & $4.92\times 10^{-12}$ \\ \hline
 $256$ & $1$ & $1.05\times 10^{-15}$ & $1.26\times 10^{-15}$ & $9.14\times 10^{-16}$ & $6.76\times 10^{-17}$ \\ \hline
 $512$ & $0$ & $1.11\times 10^{-11}$ & $6.23\times 10^{-10}$ & $6.72\times 10^{-13}$ & $6.22\times 10^{-11}$ \\ \hline
 $512$ & $1$ & $1.50\times 10^{-15}$ & $1.69\times 10^{-15}$ & $1.33\times 10^{-15}$ & $6.82\times 10^{-17}$ \\ \hline
 $1024$ & $0$ & $7.57\times 10^{-10}$ & $7.25\times 10^{-8}$ & $1.89\times 10^{-12}$ & $7.25\times 10^{-9}$ \\ \hline
 $1024$ & $1$ & $2.13\times 10^{-15}$ & $2.29\times 10^{-15}$ & $1.96\times 10^{-15}$ & $7.15\times 10^{-17}$ \\ \hline
    \end{tabular}
  \end{center}
\end{table}
\clearpage
% XY April 27 run 100.
%  3.4121e-013  1.8381e-011  1.7251e-014  1.8379e-012
%  5.4829e-013  7.2061e-012  6.0222e-014  9.0474e-013
%  2.2595e-012  4.2275e-011  2.8309e-013  4.9158e-012
%  1.1131e-011  6.2288e-010  6.7187e-013  6.2182e-011
%  7.5696e-010  7.2482e-008  1.8934e-012  7.2462e-009
%
%TableB
%  5.1047e-016  8.3007e-016  4.0222e-016  6.8583e-017
%  7.4104e-016  9.6152e-016  6.1106e-016  6.8176e-017
%  1.0635e-015  1.2626e-015  9.1437e-016  6.7616e-017
%  1.4965e-015  1.6882e-015  1.3311e-015  6.8247e-017
%  2.1321e-015  2.2857e-015  1.9578e-015  7.1473e-017  

\begin{table}[ht] 
  \caption{Residual norms rn$^{(1)}$ and the mean ratios of them and their upper bounds  $\tilde \delta_+ \ $,  
in the case of using Gaussian multipliers}
\label{tab66}
  \begin{center}
    \begin{tabular}{| c |  c | c |  c |c|}
      \hline
$q$   & n & \bf{mean} & \bf{max} & {\bf mean of ratio} {\rm rn}$^{(1)} / \tilde \Delta_+$  \\ \hline	
8 & 64 & $ 1.31\times 10^{-7}  $  &  $ 3.00\times 10^{-5} $ &   $ 1.48\times 10^{-1} $ \\ \hline		
8 & 128 & $ 1.88\times 10^{-7} $  &  $ 5.75\times 10^{-5} $ &   $ 1.52\times 10^{-1} $ \\ \hline		
8 & 256 & $ 3.84\times 10^{-7} $  &  $ 8.09\times 10^{-5} $  &  $ 1.54\times 10^{-1} $   \\ \hline		
8 & 512 &  $ 2.18\times 10^{-7} $  &  $ 2.13\times 10^{-5} $  &  $ 1.57\times 10^{-1} $   \\ \hline		
8 & 1024 &  $ 5.47\times 10^{-7} $  &  $ 2.25\times 10^{-4} $  &  $ 1.58\times 10^{-1} $   \\ \hline		
32 & 64 &  $ 5.00\times 10^{-7} $  &  $ 4.05\times 10^{-5} $  &  $ 5.23\times 10^{-2} $    \\ \hline		
32 & 128 & $ 1.98\times 10^{-6} $  &  $ 1.08\times 10^{-3} $  &  $ 6.44\times 10^{-2} $   \\ \hline		
32 & 256 & $ 1.04\times 10^{-6} $  &  $ 8.03\times 10^{-5} $  &  $ 6.90\times 10^{-2} $    \\ \hline
32 & 512 & $ 3.27\times 10^{-6} $  &  $ 1.00\times 10^{-3} $  &  $ 7.11\times 10^{-2} $   \\ \hline		
32 & 1024 & $ 3.46\times 10^{-6} $  &  $ 6.92\times 10^{-4} $  &  $ 7.30\times 10^{-2} $    \\ \hline		
    \end{tabular}
  \end{center}
\end{table}

\begin{table}[ht] 
  \caption{Residual norms rn$^{(1)}$ and the mean ratios of them and their upper bounds  $\tilde \delta_+ \ $,
 in the case of  using Toeplitz random  multipliers and Example \ref{exc1}}
\label{tab610}
  \begin{center}
    \begin{tabular}{| c | c | c | c | c | c |c|}
      \hline
$q$  &  n & \bf{mean} & \bf{max} & {\bf mean of ratio} {\rm rn}$^{(1)} / \tilde \Delta_+$ \\ \hline
8 & 64 & $ 9.70\times 10^{-8}  $  &  $ 2.01\times 10^{-5} $ &   $ 1.50\times 10^{-1} $ \\ \hline		
8 & 128 & $ 9.48\times 10^{-8} $  &  $ 6.03\times 10^{-6} $ &   $ 1.54\times 10^{-1} $ \\ \hline		
8 & 256 & $ 1.58\times 10^{-7} $  &  $ 1.17\times 10^{-5} $  &  $ 1.57\times 10^{-1} $   \\ \hline		
8 & 512 &  $ 2.77\times 10^{-7} $  &  $ 6.04\times 10^{-5} $  &  $ 1.57\times 10^{-1} $   \\ \hline		
8 & 1024 &  $ 4.97\times 10^{-7} $  &  $ 5.83\times 10^{-5} $  &  $ 1.58\times 10^{-1} $   \\ \hline		
32 & 64 &  $ 4.99\times 10^{-7} $  &  $ 5.01\times 10^{-5} $  &  $ 5.73\times 10^{-2} $    \\ \hline		
32 & 128 & $ 5.61\times 10^{-7} $  &  $ 2.43\times 10^{-5} $  &  $ 6.54\times 10^{-2} $   \\ \hline		
32 & 256 & $ 2.19\times 10^{-6} $  &  $ 7.11\times 10^{-4} $  &  $ 6.98\times 10^{-2} $    \\ \hline
32 & 512 & $ 2.53\times 10^{-6} $  &  $ 6.62\times 10^{-4} $  &  $ 7.20\times 10^{-2} $   \\ \hline		
32 & 1024 & $ 2.17\times 10^{-6} $  &  $ 3.15\times 10^{-4} $  &  $ 7.25\times 10^{-2} $    \\ \hline	 
		\end{tabular}
  \end{center}
\end{table}

\begin{table}[ht] 
  \caption{Residual norms 
rn$^{(1)}$  and the mean ratios of them and their upper bounds  $\tilde \delta_+ \ $,
in the case of  using Toeplitz random  multipliers and Example \ref{exunc}}
\label{tab613}
  \begin{center}
    \begin{tabular}{| c | c | c | c | c | c |c|}      \hline
$q$  &  n & \bf{mean} & \bf{max} & {\bf mean of ratio} {\rm rn}$^{(1)} / \tilde \Delta_+$  \\ \hline
8 & 64 & $ 1.94\times 10^{-8}  $  &  $ 3.30\times 10^{-7} $ &   $ 1.59\times 10^{-1} $ \\ \hline		
8 & 128 & $ 3.03\times 10^{-8} $  &  $ 1.97\times 10^{-6} $ &   $ 1.59\times 10^{-1} $ \\ \hline		
8 & 256 & $ 3.85\times 10^{-8} $  &  $ 1.00\times 10^{-6} $  &  $ 1.59\times 10^{-1} $   \\ \hline		
8 & 512 &  $ 5.47\times 10^{-8} $  &  $ 1.18\times 10^{-6} $  &  $ 1.59\times 10^{-1} $   \\ \hline		
8 & 1024 &  $ 8.51\times 10^{-8} $  &  $ 2.12\times 10^{-6} $  &  $ 1.59\times 10^{-1} $   \\ \hline		
32 & 64 &  $ 1.03\times 10^{-7} $  &  $ 2.84\times 10^{-6} $  &  $ 7.37\times 10^{-2} $    \\ \hline		
32 & 128 & $ 1.87\times 10^{-7} $  &  $ 2.44\times 10^{-6} $  &  $ 7.39\times 10^{-2} $   \\ \hline		
32 & 256 & $ 2.86\times 10^{-7} $  &  $ 6.43\times 10^{-6} $  &  $ 7.39\times 10^{-2} $    \\ \hline
32 & 512 & $ 4.00\times 10^{-7} $  &  $ 7.50\times 10^{-6} $  &  $ 7.38\times 10^{-2} $   \\ \hline		
32 & 1024 & $ 6.05\times 10^{-7} $  &  $ 9.54\times 10^{-6} $  &  $ 7.43\times 10^{-2} $    \\ \hline		 
		\end{tabular}
  \end{center}
\end{table}

%%% Q=8
%  1.3114e-007  2.9982e-005  1.4771e-001
%  1.8829e-007  5.7473e-005  1.5221e-001
%  3.8391e-007  8.0866e-005  1.5413e-001
%  2.1821e-007  2.1280e-005  1.5668e-001
%  5.4746e-007  2.2540e-004  1.5857e-001
%%%Q=32
%  5.0024e-007  4.0450e-005  5.2337e-002
%  1.9786e-006  1.0819e-003  6.4351e-002
%  1.0358e-006  8.0319e-005  6.9011e-002
%  3.2702e-006  1.0034e-003  7.1129e-002
%  3.4617e-006  6.9180e-004  7.2963e-002 

%------------------------------------------------------------------------------
\begin{table}[ht] 
  \caption{Residual norms rn$^{(2)}$ and the mean ratio of them and their upper bounds,  
in the case of using  Gaussian random multipliers}
\label{tab67}
  \begin{center}
    \begin{tabular}{| c |  c | c |  c |c|}
      \hline
$q$   & n & \bf{mean} & \bf{max} & {\bf mean of ratio} {\rm rn}$^{(2)} / \tilde \Delta_+$   \\ \hline	
8 & 64 & $ 2.61\times 10^{-8}  $  &  $ 5.52\times 10^{-6} $ &   $ 1.46\times 10^{-2} $ \\ \hline		
8 & 128 & $ 3.79\times 10^{-8} $  &  $ 1.21\times 10^{-5} $ &   $ 1.52\times 10^{-2} $ \\ \hline		
8 & 256 & $ 7.54\times 10^{-8} $  &  $ 1.75\times 10^{-5} $  &  $ 1.54\times 10^{-2} $   \\ \hline		
8 & 512 &  $ 4.57\times 10^{-8} $  &  $ 5.88\times 10^{-6} $  &  $ 1.55\times 10^{-2} $   \\ \hline		
8 & 1024 &  $ 1.03\times 10^{-7} $  &  $ 3.93\times 10^{-5} $  &  $ 1.56\times 10^{-2} $   \\ \hline		
32 & 64 &  $ 2.66\times 10^{-8} $  &  $ 2.02\times 10^{-6} $  &  $ 1.38\times 10^{-3} $    \\ \hline		
32 & 128 & $ 9.87\times 10^{-8} $  &  $ 5.22\times 10^{-5} $  &  $ 1.70\times 10^{-3} $   \\ \hline		
32 & 256 & $ 5.41\times 10^{-8} $  &  $ 3.52\times 10^{-6} $  &  $ 1.83\times 10^{-3} $    \\ \hline
32 & 512 & $ 1.75\times 10^{-7} $  &  $ 5.57\times 10^{-5} $  &  $ 1.89\times 10^{-3} $   \\ \hline		
32 & 1024 & $ 1.79\times 10^{-7} $  &  $ 3.36\times 10^{-5} $  &  $ 1.92\times 10^{-3} $    \\ \hline			
    \end{tabular}
  \end{center}
\end{table}

%Q=8
%  2.6134e-008  5.5161e-006  1.4573e-002
%  3.7876e-008  1.2147e-005  1.5131e-002
%  7.5415e-008  1.7491e-005  1.5372e-002
%  4.5720e-008  5.8800e-006  1.5523e-002
%  1.0341e-007  3.9259e-005  1.5563e-002
%Q=32
%  2.6571e-008  2.0163e-006  1.3785e-003
%  9.8737e-008  5.2177e-005  1.6980e-003
%  5.4096e-008  3.5176e-006  1.8296e-003
%  1.7473e-007  5.5719e-005  1.8901e-003
%  1.7875e-007  3.3632e-005  1.9238e-003  
  
%------------------------------------------------------------------------------

%------------------------------------------------------------------------------

%
%  7.9266e+000  7.1877e+000
%  5.7367e+000  2.1226e+001
%  1.2606e+001  3.6701e+000
%  5.7206e+000  1.4351e+001
%  5.1204e+000  7.8598e+000

%  9.7023e-008  2.0833e-005  1.5000e-001
%  9.4820e-008  6.0250e-006  1.5421e-001
%  1.5827e-007  1.1732e-005  1.5799e-001
%  2.7678e-007  6.0392e-005  1.5661e-001
%  4.9748e-007  5.8304e-005  1.5787e-001
%  
%  4.9920e-007  5.0066e-005  5.7298e-002
%  5.6066e-007  2.4329e-005  6.5371e-002
%  2.1898e-006  7.1097e-004  6.9763e-002
%  2.5305e-006  6.6164e-004  7.2000e-002
%  2.1688e-006  3.1494e-004  7.2520e-002

\begin{table}[ht] 
  \caption{Residual norms rn$^{(2)}$ and the mean ratio of them and their upper bounds,  
in the case of using 
 Toeplitz 
random multipliers  and Example \ref{exc1}}
\label{tab611}
  \begin{center}
    \begin{tabular}{| c | c | c | c | c | c |c|}     \hline
$q$  &  n & \bf{mean} & \bf{max} & {\bf mean of ratio} {\rm rn}$^{(2)} / \tilde \Delta_+$  \\ \hline
8 & 64 & $ 1.93\times 10^{-8}  $  &  $ 3.95\times 10^{-6} $ &   $ 1.48\times 10^{-2} $ \\ \hline		
8 & 128 & $ 1.86\times 10^{-8} $  &  $ 1.31\times 10^{-6} $ &   $ 1.52\times 10^{-2} $ \\ \hline		
8 & 256 & $ 3.24\times 10^{-8} $  &  $ 2.66\times 10^{-6} $  &  $ 1.55\times 10^{-2} $   \\ \hline		
8 & 512 &  $ 5.58\times 10^{-8} $  &  $ 1.14\times 10^{-5} $  &  $ 1.55\times 10^{-2} $   \\ \hline		
8 & 1024 &  $ 1.03\times 10^{-7} $  &  $ 1.22\times 10^{-5} $  &  $ 1.56\times 10^{-2} $   \\ \hline		
32 & 64 &  $ 2.62\times 10^{-8} $  &  $ 2.47\times 10^{-6} $  &  $ 1.52\times 10^{-3} $    \\ \hline		
32 & 128 & $ 3.00\times 10^{-8} $  &  $ 1.44\times 10^{-6} $  &  $ 1.73\times 10^{-3} $   \\ \hline		
32 & 256 & $ 1.12\times 10^{-7} $  &  $ 3.42\times 10^{-5} $  &  $ 1.84\times 10^{-3} $    \\ \hline
32 & 512 & $ 1.38\times 10^{-7} $  &  $ 3.87\times 10^{-5} $  &  $ 1.30\times 10^{-3} $   \\ \hline		
32 & 1024 & $ 1.18\times 10^{-7} $  &  $ 1.84\times 10^{-5} $  &  $ 1.92\times 10^{-3} $    \\ \hline		
    \end{tabular}
  \end{center}
\end{table}

%
%  1.9278e-008  3.9542e-006  1.4820e-002
%  1.8644e-008  1.3109e-006  1.5241e-002
%  3.2419e-008  2.6554e-006  1.5438e-002
%  5.5830e-008  1.1429e-005  1.5527e-002
%  1.0268e-007  1.2172e-005  1.5572e-002
%
%  2.6161e-008  2.4635e-006  1.5202e-003
%  2.9957e-008  1.4445e-006  1.7319e-003
%  1.1244e-007  3.4237e-005  1.8366e-003
%  1.3798e-007  3.8708e-005  1.8958e-003
%  1.1774e-007  1.8436e-005  1.9231e-003

%
%  1.9435e-008  3.2977e-007  1.5929e-001
%  3.0278e-008  1.9669e-006  1.5949e-001
%  3.8490e-008  1.0048e-006  1.5921e-001
%  5.4651e-008  1.1771e-006  1.5897e-001
%  8.5082e-008  2.1217e-006  1.5890e-001
%  	
%  1.0347e-007  2.8445e-006  7.3726e-002
%  1.8684e-007  2.4452e-006  7.3931e-002
%  2.8594e-007  6.4326e-006  7.3945e-002
%  4.0029e-007  7.4982e-006  7.3768e-002
%  6.0509e-007  9.5437e-006  7.4317e-002  

%------------------------------------------------------------------------------

\begin{table}[ht] 
  \caption{Residual norms rn$^{(2)}$ and the mean ratio of them and their upper bounds,  
in the case of using Toeplitz
random multipliers and Example \ref{exunc}}
\label{tab614}
  \begin{center}
    \begin{tabular}{| c | c | c | c | c | c |c|}     \hline
$q$  &  n & \bf{mean} & \bf{max} & {\bf mean of ratio} {\rm rn}$^{(2)} / \tilde \Delta_+$  \\ \hline			
8 & 64 & $ 3.86\times 10^{-9}  $  &  $ 1.02\times 10^{-7} $ &   $ 1.56\times 10^{-2} $ \\ \hline		
8 & 128 & $ 5.96\times 10^{-9} $  &  $ 3.42\times 10^{-7} $ &   $ 1.56\times 10^{-2} $ \\ \hline		
8 & 256 & $ 7.70\times 10^{-9} $  &  $ 2.21\times 10^{-7} $  &  $ 1.56\times 10^{-2} $   \\ \hline		
8 & 512 &  $ 1.10\times 10^{-8} $  &  $ 2.21\times 10^{-7} $  &  $ 1.56\times 10^{-2} $   \\ \hline		
8 & 1024 &  $ 1.69\times 10^{-8} $  &  $ 4.15\times 10^{-7} $  &  $ 1.56\times 10^{-2} $   \\ \hline		
32 & 64 &  $ 5.49\times 10^{-9} $  &  $ 1.61\times 10^{-7} $  &  $ 1.95\times 10^{-3} $    \\ \hline		
32 & 128 & $ 9.90\times 10^{-9} $  &  $ 1.45\times 10^{-7} $  &  $ 1.95\times 10^{-3} $   \\ \hline		
32 & 256 & $ 1.51\times 10^{-8} $  &  $ 3.05\times 10^{-7} $  &  $ 1.95\times 10^{-3} $    \\ \hline
32 & 512 & $ 2.11\times 10^{-8} $  &  $ 3.60\times 10^{-7} $  &  $ 1.95\times 10^{-3} $   \\ \hline		
32 & 1024 & $ 3.21\times 10^{-8} $  &  $ 5.61\times 10^{-7} $  &  $ 1.95\times 10^{-3} $    \\ \hline		
    \end{tabular}
  \end{center}
\end{table}

%  3.8559e-009  1.0197e-007  1.5615e-002
%  5.9550e-009  3.4193e-007  1.5618e-002
%  7.6926e-009  2.2131e-007  1.5620e-002
%  1.0913e-008  2.2131e-007  1.5621e-002
%  1.6886e-008  4.1529e-007  1.5623e-002
%  
%  5.4869e-009  1.6098e-007  1.9530e-003
%  9.9031e-009  1.4479e-007  1.9531e-003
%  1.5100e-008  3.0468e-007  1.9531e-003
%  2.1120e-008  3.6044e-007  1.9531e-003
%  3.2062e-008  5.6066e-007  1.9531e-003  
%  

\begin{table}[ht] 
  \caption{Mean ratios of the norms of the inverses of the matrices $T^T_{r,A}H$ and  
%Gaussian multipliers
 $H_{r,r}$}
\label{tab69}
  \begin{center}
    \begin{tabular}{| c | c | c |}
      \hline
 n  &  \bf{$q=8$} & \bf{$q=32$} \\ \hline	
 64   & $ 7.93 $  &  $ 7.19 $   \\ \hline				
 128  & $ 5.74 $  &  $ 2.12 $   \\ \hline				
 256  & $ 1.26 $  &  $ 3.67 $  \\ \hline				
  512 & $ 5.72 $  &  $ 1.44 $    \\ \hline		
 1024 & $ 5.12 $  &  $ 7.86 $   \\ \hline		
    \end{tabular}
  \end{center}
\end{table}

%------------------------------------------------------------------------------

\begin{table}[ht] 
  \caption{Condition numbers of Gaussian matrices}
\label{tab616}
  \begin{center}
    \begin{tabular}{| c | c | c | c | c |}
      \hline
 n & \bf{mean} & \bf{max} & \bf{min} & \bf{std} \\ \hline	
 64 &  $ 1.83 $  &  $ 2.47 $  &  $ 1.40  $ & $ 0.16 $  \\ \hline		
 128 &  $ 1.51 $ &  $ 1.77 $  &  $ 1.30  $ & $ 0.08 $  \\ \hline		
 256 &  $ 1.34 $ &  $ 1.55 $  &  $ 1.20  $ & $ 0.05 $  \\ \hline		
 512 &  $ 1.23 $ &  $ 1.38 $  &  $ 1.11  $ & $ 0.03 $  \\ \hline		
 1024&  $ 1.15 $ &  $ 1.23 $  &  $ 1.08  $ & $ 0.02 $  \\ \hline		
    \end{tabular}
  \end{center}
\end{table}

\begin{table}[ht] 
  \caption{Condition numbers of circulant matrices of Example \ref{exc1}}
\label{tab617}
  \begin{center}
    \begin{tabular}{| c | c | c | c | c |}
      \hline
 n & \bf{mean} & \bf{max} & \bf{min} & \bf{std} \\ \hline	
  64 &  $ 4.65\times 10^{+1} $  &  $ 6.66\times 10^{+3} $  &  $ 4.11\times 10^{+0} $ & $ 2.91\times 10^{+2} $    \\ \hline		
 128 &  $ 4.91\times 10^{+1} $  &  $ 3.93\times 10^{+3} $  &  $ 5.92\times 10^{+0} $ & $ 1.65\times 10^{+2} $   \\ \hline		
 256 &  $ 1.40\times 10^{+2} $  &  $ 7.31\times 10^{+4} $  &  $ 8.50\times 10^{+0} $ & $ 2.32\times 10^{+3} $   \\ \hline		
 512 &  $ 1.01\times 10^{+2} $  &  $ 1.06\times 10^{+4} $  &  $ 1.33\times 10^{+1} $ & $ 4.69\times 10^{+2} $   \\ \hline	
1024 &  $ 1.16\times 10^{+2} $  &  $ 3.48\times 10^{+3} $  &  $ 1.97\times 10^{+1} $ & $ 1.79\times 10^{+2} $   \\ \hline		
    \end{tabular}
  \end{center}
\end{table}

%------------------------------------------------------------------------------

\end{document}